\documentclass[reqno,11pt]{amsart}
\usepackage{latexsym}
\usepackage{amsfonts}
\usepackage{amssymb}
\usepackage{mathrsfs}
\usepackage[all]{xy}
\usepackage{bbm}
\usepackage{xcolor}
\usepackage{ulem}
\usepackage[mathscr]{euscript}

\newtheorem{theorem}{Theorem}[section]
\newtheorem{corollary}[theorem]{Corollary}
\newtheorem{lemma}[theorem]{Lemma}
\newtheorem{conjecture}[theorem]{Conjecture}
\newtheorem{proposition}[theorem]{Proposition}
\theoremstyle{definition}
\newtheorem{example}[theorem]{Example}%[section]

\newtheorem{remark}[theorem]{Remark}
\numberwithin{equation}{section}

\newtheorem{innercustomthm}{Theorem}
\newenvironment{customthm}[1]
  {\renewcommand\theinnercustomthm{#1}\innercustomthm}
  {\endinnercustomthm}

\newcommand{\RR}{\mathbb{R}}
\newcommand{\CC}{\mathbb{C}}
\newcommand{\NN}{\mathbb{N}}

\newcommand{\ZZ}{\mathbb{Z}}

\newcommand{\TT}{\mathbb{T}}
\renewcommand{\SS}{\mathbb{S}}
\newcommand{\mt}{\mapsto}

\newcommand{\cN}{\mathcal{N}}

\newcommand{\OO}{\mathcal{O}}
\newcommand{\OOO}{\mathcal{\widehat{O}}}

\newcommand{\sB}{\mathscr{B}}
\newcommand{\sL}{\mathscr{L}}

\newcommand{\sW}{\mathscr{W}}

\newcommand{\sN}{\mathscr{N}}

\newcommand{\bx}{\mathbf{x}}
\newcommand{\by}{\mathbf{y}}
\newcommand{\bz}{\mathbf{z}}
\newcommand{\bu}{\mathbf{u}}
\newcommand{\bv}{\mathbf{v}}
\newcommand{\be}{\mathbf{e}}
\newcommand{\bw}{\mathbf{w}}

\newcommand{\diag}{\operatorname{diag}}
\newcommand{\Sp}{\operatorname{Sp}}
\newcommand{\id}{\operatorname{id}}
\newcommand{\dist}{\operatorname{dist}}

\newcommand{\GL}{\operatorname{GL}}
\newcommand{\Ma}{\operatorname{M}}
\newcommand{\SL}{\operatorname{SL}}
\newcommand{\SO}{\operatorname{SO}}
\newcommand{\Ad}{\operatorname{Ad}}
\newcommand{\Gr}{\operatorname{Gr}}

\newcommand{\Lie}{\operatorname{Lie}}
\newcommand{\Aut}{\operatorname{Aut}}

\newcommand{\Stab}{\operatorname{Stab}}
\renewcommand{\ker}{\operatorname{Ker}}
\newcommand{\Id}{\operatorname{Id}}
\newcommand{\im}{\operatorname{Im}}
\newcommand{\B}{\mathscr{B}}
\newcommand{\E}{\mathscr{E}}

\newcommand{\Lf}{\mathfrak{f}}
\newcommand{\Lg}{\mathfrak{g}}
\newcommand{\Lh}{\mathfrak{h}}
\newcommand{\La}{\mathfrak{a}}

\newcommand{\Lp}{\mathfrak{p}}
\newcommand{\Lk}{\mathfrak{k}}
\newcommand{\la}{\langle}
\newcommand{\ra}{\rangle}
\newcommand{\pu}{\pmb{u}}
\newcommand{\pv}{\pmb{v}}
\newcommand{\pw}{\pmb{w}}

\newcommand{\T}{\textsf{T}}
\newcommand{\bin}{\textstyle\binom}

\newcommand {\ignore}[1]  {}

\newif\ifdraft\drafttrue
\draftfalse

\newcommand\eq[2]{{\ifdraft{\ \tt [#1]}\else\ignorespaces\fi}\begin{equation}\label{eq:#1}{#2}\end{equation}}

\newcommand {\equ}[1]     {\eqref{eq:#1}}
\newcommand{\ggm}{G/\Gamma}

\newcommand\hd{Hausdorff dimension}

\newcommand\hs{homogeneous space}

\newcommand{\x}{{\bf x}}

\newcommand{\sm}{\smallsetminus}
\newcommand{\smz}{\smallsetminus\{0\}}
\renewcommand{\emptyset}{\varnothing}
\renewcommand{\setminus}{\smallsetminus}

\newcommand{\DDT}{\textstyle\left.\frac{d}{dt}\right|_{t=0}}
\newcommand{\DT}{\textstyle\left.\frac{d}{dt}\right|_{t=1}}

\begin{document}

\title[Nondense orbits on homogeneous spaces  and applications]
{Nondense orbits on homogeneous spaces \\ and applications to geometry and number theory}

\author{Jinpeng An}
\address{School of Mathematical Sciences, Peking University, Beijing, 100871, China}
\email{anjinpeng@gmail.com}

\author{Lifan Guan}
\address{Mathematisches Institut, Georg-August Universit{\"a}t G{\"o}ttingen, Bunsenstrasse 3-5, D-37073 Gottingen, Germany}
\email{guanlifan@gmail.com}

\author{Dmitry Kleinbock}
\address{Department of Mathematics, Brandeis University, Waltham MA 02454-9110, USA}
\email{kleinboc@brandeis.edu}

\date{October 2020}
\thanks{Research supported by an NSFC grant (J.A.), by ERC Consolidator grant 648329 (L. G.) and by NSF grants %DMS-1101320,
DMS-1600814 and DMS-1900560 (D.K.)}

\begin{abstract}
Let $G$ be a  Lie group, $\Gamma\subset G$  a discrete subgroup,  $X=G/\Gamma$, and $f$ an affine map from $X$ to itself. We give conditions on a submanifold $Z$ of $X$ guaranteeing that the set of points $x\in X$ with $f$-trajectories avoiding $Z$ is hyperplane absolute winning (a property which implies full \hd\ and is stable under countable intersections). A similar result is proved for one-parameter actions on $X$. This has applications to constructing exceptional geodesics on locally symmetric spaces, and to non-density of the set of values of certain functions at integer points.
\end{abstract}

\maketitle

\section{Introduction}

\subsection{Nondense orbits in homogeneous dynamics.}

Let $X$ be a metric space and $F$ a set of self-maps $X \rightarrow X$. For a non-empty subset $Z$ of $X$, define
$$E(F,Z) :=  \left\{x\in X: \overline{Fx} \cap Z =\emptyset\right\}.$$
When $f$ is a single transformation of $X$ we will slightly abuse notation and define
$$E(f,Z) := E(\{f^n: {n\ge0}\},Z) =  \left\{x\in X: \overline{\{f^nx: {n\ge0}\}} \cap Z =\emptyset\right\}.$$
Those sets carry important information about the dynamical system $(X,F)$ and have been extensively studied. Clearly one has $\mu\big(E(f,Z)\big) = 0$ whenever $\mu$ is an $f$-ergodic measure on $X$ with full support. On the other hand, for certain classes of dynamical systems and subsets $Z$ of $X$, sets of those exceptional points can be shown to be quite substantial -- in particular, they are thick. Here and hereafter we say that $E\subset X$ is  {\sl thick\/}  if $\dim(U\cap E) = \dim(U)$ for any open subset $U$ of $X$, where $\dim$ stands for the \hd. See e.g.\ \cite{U, KM, Do, Kl2} for some work done in this direction in the 1990s.

In this paper, we restrict ourselves to the study of the special case when $X$ is a homogeneous space of a Lie group $G$.
Around $25$ years ago, the  {third-named} author considered the case when
$F$ is either a one-parameter or a cyclic semigroup of $G$ acting on $X$ by left translations.
To state this result we need to define  the  {\sl expanding horospherical subgroup} $G_f$ corresponding to $f\in G$:
$$G_f:=\left\{g\in G: \lim_{n\rightarrow \infty} f^{-n}gf^n=1_G \right\}.$$
Another way of defining  $G_f$ is as follows: its Lie algebra is the subalgebra of $\Lie(G)$ whose complexification is the
direct sum of generalized eigenspaces of $\Ad f$ corresponding to eigenvalues of modulus greater than $1$.
See \S\ref{affine} for a discussion. If $F  = \{g_t : t\in \RR\}$ is a   one-parameter subgroup of $G$, we will  denote  $F^+ := \{g_t : t\ge 0\}$ and $F^- := \{g_t : t\le 0\}$, and define the {\sl expanding horospherical subgroup} $G_{F^\pm}$ corresponding to $F^\pm$ as \eq{ehs}{G_{F^\pm}:= G_{g_{{}_{\pm1}}} = \left\{g\in G: \lim_{n\rightarrow \infty} g_{\mp n}gg_{\pm n}=1_G\right \}.}

When $Z\subset X$ is a smooth submanifold, it turns out that a condition sufficient for abundance of orbits avoiding $Z$ can be phrased in the language of transversality.  Let $G$ be a Lie group, $D\subset G$  a closed subgroup (not necessarily discrete),  $X=G/D$,  and let $H,F$ be Lie subgroups of $G$. According to the terminology introduced in  \cite{Kl2,KW3}, a $C^1$ submanifold  $Z$ of $X$ is said to be
 \begin{itemize}
 \item {\sl $H$-transversal} if $T_z(Hz)\not\subset T_zZ$ for every $z\in Z$;
 \smallskip
 \item {\sl $(F,H)$-transversal} if $T_z(Fz)\not\subset T_zZ$ for every $z\in Z$ (that is,  $Z$ is $F$-transversal), and also $T_z(Hz)\not\subset T_zZ\oplus T_z(Fz)$ for every $z\in Z$.
 \end{itemize}
The following theorem was proved in \cite{Kl2}:

\begin{theorem}\label{T:Kle}
 Let $G$ be a Lie group, $\Gamma\subset G$  a discrete subgroup, and $X=G/\Gamma$.
\begin{itemize}
\item[(1)]  Let $f\in G$. Then for any compact {$G_f$-transversal} $C^1$ submanifold $Z\subset X$ the set $E(f,Z)$ is thick.
 \smallskip
\item[(2)]  Let $F\subset G$ be a one-parameter subgroup.  Then for any compact $(F,G_{F^+})$-transversal $C^1$ submanifold $Z\subset X$  the set $E(F^+,Z)$ is thick.
\end{itemize}
\end{theorem}

We note that   the above theorem is meaningful only if $\Ad f $ (resp., $\Ad g_1$) has at least one eigenvalue of modulus $>1$; otherwise the groups $G_f$ (resp., $G_{F^+}$) are trivial, and the above transversality conditions are never satisfied.

\smallskip

The abundance of points with non-dense orbits has also been established when \linebreak $f\in \GL_n(\RR)\cap \Ma_{n\times n}(\ZZ)$ is an endomorphism of the $n$-dimensional torus. Indeed, generalizing a result of Dani  \cite{Da3}, Broderick, Fishman and Kleinbock \cite{BFK} proved the following:
\begin{theorem} \label{T:BFS}
  Let $X=\TT^n$, and let $f\in \GL_n(\RR)\cap \Ma_{n\times n}(\ZZ)$ be an endomorphism of $X$ with at least one eigenvalue of modulus bigger than $1$. Then for any countable subset $Z\subset X$, the set $E(f,Z)$ is  thick.
\end{theorem}

In fact, both in  \cite{Da3} and in \cite{BFK} a stronger property of those sets was established: namely, they were shown to be winning in the sense of Schmidt. Later this property was upgraded by Broderick, Fishman and Simmons \cite{BFS} to an even stronger hyperplane absolute winning property (abbreviated as HAW).  See \cite{BFKRW, KW3}, as well as  \S\ref{haw}, for definitions and discussion,   and \cite{KW1, HY, AGK, Ts, Wu1, Wu2,  GW, AGGL, Du} for other recent results involving winning properties of exceptional sets in dynamical systems. We point out  that one of the important advantages of this strengthening is the fact that a countable intersection of winning (resp., HAW) sets is also winning (resp., HAW).

\smallskip

Our first main theorem (Theorem \ref{A1} below) gives a unified treatment of Theorem  \ref{T:BFS} and part (1) of Theorem  \ref{T:Kle}. To include both left translations on \hs s and toral endomorphisms, we establish our result for affine maps. Let $G$ be a Lie group (not necessarily connected) with Lie algebra $\Lg$, $\Gamma\subset G$  a discrete subgroup, and $X=G/\Gamma$. Let $\Aut(G,\Gamma)$ denote the set of automorphisms of $G$ sending $\Gamma$ into $\Gamma$. A map $f:X\to X$ is said to be {\sl affine} if there exist $g\in G$ and $\sigma\in\Aut(G,\Gamma)$ such that
\eq{aff}{f(h\Gamma)=g\sigma(h)\Gamma, \qquad \forall \, h\in G.}
Let $\sigma_f$ be the automorphism of $G$ given by
\eq{auto}{\sigma_f(h)=g\sigma(h)g^{-1}, \qquad h\in G,}
and let $d\sigma_f$ be the induced automorphism of $\Lg$.
(It will be shown that $d\sigma_f$ is uniquely determined by $f$, see Lemma \ref{L:iterations}.)
In \S\ref{main} for an affine map $f$ we, similarly to \equ{ehs}, define  the {\sl  expanding horospherical subgroup} $G_f$ of $G$ relative to $f$, and also introduce a subgroup $G_f^{\max}\subset G_f$, which we call the {\sl maximally expanding horospherical subgroup} of $G$ relative to $f$. Roughly speaking, the latter subgroup corresponds to directions in $\Lg$ in which $d\sigma_f$ exhibits the maximal rates of expansion. For example,  if $d\sigma_f$ is diagonalizable with at least one eigenvalue of modulus bigger than $1$, then the Lie algebra of $G_f^{\max}$ is
the sum of eigenspaces corresponding to eigenvalues of $d\sigma_f$ with maximal absolute value. See \S\ref{poly} for a formal approach,  and \S\ref{affine} for a precise definition. This subgroup replaces $G_f$ in the transversality conditions of Theorem \ref{T:Kle}, which makes it possible to upgrade its conclusion to the winning property of $E(f,Z)$, as follows:

\begin{customthm}{A1}\label{A1}
\textit{Let $G$ be a Lie group, $\Gamma\subset G$ a discrete subgroup, $X=G/\Gamma$, and $f$ an affine map on $X$. Then for any $G_f^{\max}$-transversal $C^1$ submanifold $Z\subset X$ the set $E(f,Z)$ is HAW.}
\end{customthm}

We remark that if $f$ is an affine map and the assumption of $G_f^{\max}$-transversality of $Z$ is replaced by a weaker assumption of $G_f$-transversality, it is possible to use the methods of \cite{Kl2} to show that the set $E(f,Z)$ is thick. However in order to prove the HAW property  (or even  regular winning in the sense of Schmidt) $G_f$-transversality does not seem to be enough, and one has to require transversality with respect to $G_f^{\max}$.
\smallskip

In the case $G=\RR^n$, $\Gamma=\ZZ^n$ and $f\in \GL_n(\RR)\cap \Ma_{n\times n}(\ZZ)$, we note that $d\sigma_f$ coincides with $f$, and Theorem \ref{A1} implies the following strengthening of Theorem \ref{T:BFS} and the subsequent work in \cite{BFS}:
\begin{corollary}
  Let $X=\TT^n$, and let $f\in \GL_n(\RR)\cap \Ma_{n\times n}(\ZZ)$ be an endomorphism of $X$.
  Then for any $G_f^{\max}$-transversal $C^1$ submanifold $Z\subset \TT^n$ the set $E(f,Z)$ is HAW.
\end{corollary}
As an example: if $f$ is uniformly expanding, such as $\x\mapsto m\x$ for a nonzero $m\in\ZZ$, then $G_f^{\max} = \RR^n$, and thus $E(f,Z)$ is hyperplane absolute winning for any  proper $C^1$ submanifold  $Z\subset \TT^n$.

\smallskip

The case $\sigma = \Id$ of Theorem \ref{A1}  (that is, when $f$ is a left translation by an element $g$ of $G$) can be used to derive a continuous version of the above theorem, that is, a statement similar to part (2) of Theorem  \ref{T:Kle}. Here we will denote by $G_{F^+}^{\max}$ the maximally expanding horospherical subgroup $G_{g_1}^{\max}$  of $G$ relative to $g_1$.

\begin{customthm}{A2}\label{A2}
\textit{Let $G$, $\Gamma$ and  $X$ be as in Theorem \ref{A1}. Let $F=\{g_t:t\in\RR\}$ be a one-parameter subgroup of $G$, and let $Z$ be an $(F,G_{F^+}^{\max})$-transversal $C^1$ submanifold of $X$. Then the set $E(F^+,Z)$ is HAW.}
\end{customthm}

Note that, in view of intersection properties of winning sets, the conclusion of the two theorems above will hold if $Z$ is replaced by a countable union of sets satisfying the above assumptions. Note also that the groups   $G_f^{\max}$ (resp., $G_{F^+}^{\max}$) are non-trivial if and only if $d\sigma_f$ (resp., $\Ad g_1$) has at least one eigenvalue of modulus $>1$. In the latter case  the transversality conditions in Theorems \ref{A1} and \ref{A2}
are definitely satisfied if $Z$ consists of a single point, and hence the conclusion of the two theorems holds
for countable sets $Z$.

\subsection{Nondense geodesics on locally symmetric spaces.}
Theorems \ref{A1} and \ref{A2} will be derived from their more general technical versions, Theorems \ref{T:main-discrete} and \ref{T:main-continuous}, where we study the HAW property of the intersections of the sets $E(f,Z)$ and  $E(F^+,Z)$ with orbits of
certain subgroups $H \subset G$. The advantage of such a general set-up is that some important applications can be deduced from it. In particular, when $G$ is semisimple and $H$ is taken to be the maximal compact subgroup of $G$, Theorem \ref{T:main-continuous} has interesting applications to geodesic flows on locally symmetric spaces.

\smallskip

Let $Y$ be a locally symmetric space of noncompact type, and let $S(Y)$ denote its unit tangent bundle, whose fiber $S_y(Y)$ over a point $y\in Y$ is the unit sphere in $T_yY$ centered at the origin. For $\xi\in S(Y)$, let $\gamma(\xi)$ denote the geodesic line through the base point of $\xi$ in the direction $\xi$. We will use Theorem \ref{T:main-continuous} to prove the following result.

\begin{customthm}{B1}\label{B1}
\textit{Let $Y$ be a locally symmetric space of noncompact type, $y\in Y$, and $Z$ a countable subset of $Y\sm\{y\}$. Then the set
$$\left\{\xi\in S_y(Y): \overline{\gamma(\xi)}\cap Z=\emptyset\right\}$$
  is thick in $S_y(Y)$.}
\end{customthm}

Theorem \ref{B1}, together with Marstrand's slicing theorem, implies that for any countable subset $Z\subset Y$, the set $\{\xi\in S(Y): \overline{\gamma(\xi)}\cap Z=\emptyset\}$ is thick in $S(Y)$. For locally symmetric spaces of constant negative curvature (which corresponds to the case $G=\SO(n,1)$), the latter result for finite $Z$ is given in \cite[Corollary 4.4.4]{Kl2}; see also a related work by Dolgopyat \cite{Do}.

\smallskip

Note that if $Y$ has rank one and has finite volume, the geodesic flow on $S(Y)$ is ergodic. However, it is never ergodic if the rank of $Y$ is greater than one. Mautner \cite{Mau} showed that $S(Y)$ can be naturally partitioned into closed submanifolds that are invariant under the geodesic flow (see also \cite{KM2}). If $Y$ has finite volume, the geodesic flow is ergodic on a generic submanifold. We refer to a submanifold in this partition as an \textsl{ergodic submanifold} (see \S \ref{S:geo} for definition). We will also prove the following theorem.

\begin{customthm}{B2}\label{B2}
\textit{Let $Y$ be a locally symmetric space of noncompact type, $\E\subset S(Y)$ an ergodic submanifold, and $\,\Xi\subset\E$ a finite subset. Then
there exists a closed subset of $\E$ that is invariant under the geodesic flow, does not intersect $\Xi$, and projects onto $Y$.}
\end{customthm}

Theorem \ref{B2} is motivated by the unpublished work of Burns and Pollicott \cite{BP1} and subsequent papers \cite{BSW,Schro,BS,Rei}, where hyperbolic manifolds and more general manifolds of nonpositive curvature are considered. However, in all the aforementioned papers the set
$\Xi$ consisted of a single point.  Theorem \ref{B2} seems to be new even in the case when $Y$ has rank one (in which case one has $\E=S(Y)$). See also \cite{Weil} for a related work concerning CAT(-1) spaces.

\subsection{Gaps between values of functions at integer points.}\label{gaps}

For the special case \linebreak $\ggm=\SL_2(\RR)/\SL_2(\ZZ)$, Theorem  \ref{A2} was already established in  \cite{KW3} by the third-named author and Weiss. That paper was in fact motivated by studying binary indefinite quadratic forms with  non-dense set of values at integer points, and contains the following result:

\begin{theorem}\label{T:KW}
The set of indefinite binary quadratic forms whose set of values at nonzero integer points misses a given countable set is thick in the space of all binary indefinite quadratic forms.
\end{theorem}

More generally, given $\phi\in C(\RR^n)$ one can consider the $\SL_n(\RR)$-orbit of $\phi$:
\eq{orbit}{\OO(\phi):=\{\phi\circ g:g\in{\SL_n(\RR)}\}.}
Then the stabilizer $\Aut(\phi)$ of $\phi$ is a closed subgroup of $\SL_n(\RR)$, and hence the orbit $\OO(\phi)\cong \Aut(\phi)\backslash\SL_n(\RR)$ has a natural smooth manifold structure.
Theorem \ref{T:KW} dealt with the case $n=2$ and $\phi(x_1,x_2) = x_1x_2$.
See \S \ref{forms} for more background on this problem.

\smallskip

Using Theorem \ref{A2}, we are able to prove a substantial generalization of Theorem \ref{T:KW}. Let $n\ge 2$, and fix a norm $\|\cdot\|$ on $\RR^n$. We say that a continuous function $\phi: \RR^n \rightarrow \RR$ is a  \textsl{generalized indefinite binary form}, abbreviated as {\sl GIBF},
if there exists a nontrivial decomposition $\RR^n=U\oplus W$ such that the following three conditions hold:
\begin{itemize}
  \item[(IB-1)]  $\phi$ is invariant under the one-parameter transformation group
\begin{equation}\label{E:gt}
F=\{g_t:t\in\RR\}, \ \text{ where } \ g_t=e^{t/p}\,\id_{U}\oplus e^{-t/q}\,\id_{W}, \ p=\dim U, \ q=\dim W.
\end{equation}
  \item[(IB-2)] $\phi(0)=0$, and there is a continuous function $N:\RR\to[0,\infty)$ with $N(0)=0$ such that
  $$N\big(\phi(\pmb{u}+\pmb{w})\big)\ge\|\pmb{u}\|^p\|\pmb{w}\|^q,  \qquad \forall \, \pmb{u}\in U,\ \pmb{w}\in W.$$
  \item[(IB-3)]  For any $a\ne 0$, the set $\phi^{-1}(a)$ is contained in a countable union of  $F$-invariant $C^1$ submanifolds of $\RR^n$ that are both $U$-transversal and $W$-transversal\footnote{Here the transversality is understood in the sense of the action of $U$ and $W$ on $\RR^n$ by translations.}.
\end{itemize}
It is clear that the above property is preserved by linear changes of coordinates, and thus if $\phi$ is a GIBF, its orbit \equ{orbit} consists entirely of GIBFs.
{The binary form $x_1x_2$ mentioned above is clearly a GIBF, {and thus the same is true for all indefinite binary forms}. The polynomials listed below are also GIBF{s}:
\begin{align}
& x_1(x_2^{n-1}+\cdots+x_n^{n-1}), && n \text{ odd}; \label{E:poly1}\\
& (x_1^2+x_2^2)(x_3^{n-2}+\cdots+x_n^{n-2}), && n \text{ even}; \label{E:poly2}\\
& (x_1^2+\cdots+x_{n/2}^2)(x_{(n/2)+1}^2+\cdots+x_n^2), && n \text{ even}; \label{E:poly3}\\
& x_1x_2^2+x_1^3x_3^6 , && n=3. \label{E:poly4}
\end{align}
For the verification of the above claim and for more examples of GIBFs, see \S\ref{Examples}.}
As a non-polynomial example, if the norm $\|\cdot\|$ is $C^1$ on $\RR^n\smallsetminus(U\cup W)$, then the function
$$\phi(\pmb{u}+\pmb{w}) = \|\pmb{u}\|^p\|\pmb{w}\|^q, \qquad \pmb{u}\in U,\ \pmb{w}\in W$$
is a GIBF (see Example \ref{example1}).

\smallskip
Now we are ready to generalize Theorem  \ref{T:KW} to the set-up of   gaps between values of these functions at nonzero integer points:

\begin{customthm}{C}\label{C}
\it  Let $n\ge 2$, and let $\phi$ be a GIBF. Then for any countable subset $A$ of $\RR$, the set
$$\left\{\psi\in\OO(\phi):\overline{\psi(\ZZ^n\smz)}\cap A=\emptyset\right\}$$
is hyperplane absolute winning.
\end{customthm}

\subsection{Organization of the paper.}

In \S \ref{main}  we state our main technical results, Theorem  \ref{T:main-discrete} and Theorem \ref{T:main-continuous}, and deduce the latter from the former. Section \ref{hyp} is devoted to the study of  behavior of certain hyperplanes under linear transformations, which is utilized in the subsequent section for the proof of  Theorem \ref{T:main-discrete}. There we use the hyperplane percentage game, a modification of the hyperplane absolute game introduced   in \cite{BFS} (see \S\ref{S:HPW}), a careful analysis of the local behavior of the multiplication on $G$ (\S\ref{lie}), and an approximation of pieces of submanifolds $Z$ by neighborhoods of hyperplanes (\S\ref{z}). In \S\ref{S:geo} we apply Theorem \ref{T:main-continuous} to geodesic flows on locally symmetric spaces, proving  Theorems \ref{B1} and \ref{B2}.  Then in \S\ref{forms} we discuss another application, in which we put $X = \SL_n(\RR)/\SL_n(\ZZ)$ and establish a general result (Theorem \ref{T:app1-main}) concerning  functions whose values at integer points are not dense.  Theorem \ref{C} is derived from Theorem \ref{T:app1-main} in \S\ref{gibfs}, and then we describe  a number of examples of generalized indefinite binary forms.

\section{Statement of the main theorems}\label{main}

\subsection{HAW subsets of a manifold}\label{haw}

Our main theorems are stated in terms of the notion of  hyperplane absolute winning (HAW) subsets of smooth manifolds introduced in \cite{KW3}.
Before defining this game, for comparison let us
recall Schmidt's $(\alpha,\beta)$-game  \cite{Sc1}. It
involves two parameters $\alpha,\beta\in(0,1)$  and is played by two players Alice and Bob on a Euclidean space $V$ with a target set $S\subset V$. Bob starts the game by choosing a closed ball $B_0$ in $V$ with center $x_0$ and radius $r_0$. After Bob chooses a closed ball $B_i = B(x_i,{r}_i) $, Alice chooses  $A_i = B(x_i', {r}_i'){\subset B_i}$ with ${r}'_i=\alpha {r}_i$,
 and then Bob chooses  $B_{i+1} = B(x_{i+1},{r}_{i+1}) {\subset A_i}$ with ${r}_{i+1}=\beta {r}'_{i} $
etc. Alice wins the game if the unique point of intersection
$$\bigcap_{i=0}^\infty A_i=\bigcap_{i=0}^\infty B_i$$ belongs to $S$, and Bob wins otherwise. The set $S$ is \textsl{$(\alpha,\beta)$-winning} if Alice has a winning strategy, is \textsl{$\alpha$-winning} if it is $(\alpha,\beta)$-winning for any $\beta\in(0,1)$, and is \textsl{winning} if it is $\alpha$-winning for some $\alpha$. Schmidt \cite{Sc1} proved that winning sets are thick, and that  a countable intersection of $\alpha$-winning sets is again $\alpha$-winning.

A more recent  development of the theory started with a paper of McMullen  \cite{Mc} who introduced the notion of {\sl absolute winning sets}. Those were   generalized in \cite{BFKRW} to {\sl $k$-dimensionally absolute winning} for any $0 \le k < \dim V$. In particular,  the hyperplane absolute game (the case $k = \dim V - 1$) is played on an open subset $U$ of  $V$ as follows. Again, there are two players called  Alice and Bob, and a target set $S \subset U$. Let
$\beta \in \left(0, \frac13 \right)$; Bob starts the game by choosing a closed Euclidean ball $B_0$ contained in $U$ of radius ${r}_0$.
For an affine hyperplane $L\subset V$ and $r>0$, we denote the ${r}$-neighborhood of $L$ by
$$L^{({r})}:=\{v\in V:\mathrm{dist}(v,L)<{r}\}.$$
After Bob chooses a closed Euclidean ball $B_i\subset U$ of radius ${r}_i$, Alice chooses a hyperplane neighborhood $L_i^{({r}'_i)}$ with ${r}'_i\le\beta {r}_i$,
and then Bob chooses a closed ball $B_{i+1}\subset B_i\setminus L_i^{({r}'_i)}$ of radius ${r}_{i+1}\ge\beta {r}_i$. Alice wins the game if
$$\bigcap_{i=0}^\infty B_i\cap S\ne\emptyset.$$ The set $S$ is   \textsl{$\beta$-hyperplane absolute winning   on $U$}, abbreviated as $\beta$-HAW,   if Alice has a winning strategy,
and is \textsl{HAW on $U$} if it is $\beta$-HAW for any $\beta\in(0,\frac{1}{3})$. It is easy to see that HAW sets are winning in the sense of Schmidt. Moreover, it is proved in \cite{BFKRW} that the property of being hyperplane absolute winning is invariant under $C^1$ diffeomorphisms: if $\varphi:U\to V$ is a $C^1$ diffeomorphism onto an open subset $\varphi(U)$ of $V$, then $S$ is HAW on $U$ if and only if $\varphi(S)$ is HAW on $\varphi(U)$. In particular, the class of HAW sets is independent of the inner product on $V$.

\smallskip
The aforementioned property, as shown in   \cite{KW3},  can be used to define the notion of HAW sets for subsets of $C^1$  manifolds. Namely, let $M$ be a $C^1$ manifold, and let $\{(U_\alpha, \varphi_\alpha)\}$ be a $C^1$ atlas, that is, $\{U_\alpha\}$ is an open cover of $M$, and each $\varphi_\alpha$ is a $C^1$ diffeomorphism from $U_\alpha$ onto the open subset $\varphi_\alpha(U_\alpha)$ of a Euclidean space $V$. A subset $S\subset M$ is said to be \textsl{HAW} if for each $\alpha$, $\varphi_\alpha(S\cap U_\alpha)$ is HAW on $\varphi_\alpha(U_\alpha)$. The $C^1$ invariance implies that the definition is independent of the choice of the atlas. Moreover,
we can summarize the above discussion as follows:
\begin{itemize}
\item  HAW subsets of a $C^1$ manifold are thick;
\item  a countable intersection of HAW subsets of a $C^1$ manifold is again HAW;
\item let $\varphi:M\to N$ be a diffeomorphism between $C^1$ manifolds; then $S\subset M$ is HAW if and only if $\varphi(S)$ is HAW.
\end{itemize}

For the proof of Theorems \ref{B1} and \ref{C} we will also need the following lemma.

\begin{lemma}\label{L:haw-proj}
Let $\varphi:M\to N$ be a surjective $C^1$ submersion of $C^1$ manifolds. If $S\subset M$ is HAW, then so is $\varphi(S)\subset N$.
\end{lemma}

Note that this lemma is complementary to \cite[Proposition 6.1]{HKS}, where preimages of HAW sets under   surjective $C^1$ maps are considered. We postpone the proof of Lemma \ref{L:haw-proj} to Appendix \ref{S:A}.

\subsection{A polynomial associated with a linear transformation}\label{poly}

Let $V$ be a finite-dimensional real vector space, regarded as a real subspace of its complexification \linebreak $V_\CC:=V\otimes_\RR\CC$.
For a linear transformation $T$ on $V$, let $T_\CC$ denote the complex linear extension on $V_\CC$. Let $\Sp(T)$ be the set of eigenvalues of $T_\CC$, and let $$\rho=\rho(T)=\max_{\lambda\in\Sp(T)}|\lambda|$$
be the spectral radius of $T$.  Let
$$p_0(x)=\prod_{\lambda\in\Sp(T)}(x-\lambda)^{s(\lambda)}$$
be the minimal polynomial of $T$, and define
$$s=s(T)=\max_{\lambda\in\Sp(T),|\lambda|=\rho}s(\lambda).$$
The polynomial $p(x)$ given in the following lemma will play an important role.

\begin{lemma}\label{L:polynomial}
There exists a unique real polynomial $p(x)=p_T(x)$ with $\deg p(x)<\deg p_0(x)$ such that for every $\lambda\in\Sp(T)$, we have
$$p(x)\equiv\begin{cases}
  (x-\lambda)^{s-1}, & \text{if $|\lambda|=\rho$ and $s(\lambda)=s$;}\\
  0, & \text{otherwise,}
\end{cases} \mod (x-\lambda)^{s(\lambda)}.$$
\end{lemma}

\begin{proof}
The existence and uniqueness of a complex polynomial $p(x)$ satisfying the required properties follow directly from the Chinese Remainder Theorem. Since the minimal polynomial $p_0(x)$ is real, we have $s(\bar{\lambda})=s(\lambda)$ for every $\lambda\in\Sp(T)$. Thus the complex conjugate of $p(x)$ also satisfies the requirement, and hence, the uniqueness implies that $p(x)$ is indeed real.
\end{proof}

We will need to consider the transformation $p(T)$. To understand it, let us consider the Jordan normal form of $T_\CC$. Let
\begin{equation}\label{E:basis}
\B=\{\be_{11},\ldots,\be_{1,s_1},\be_{21},\ldots,\be_{2,s_2},\ldots,\be_{r1},\ldots,\be_{r,s_r}\}
\end{equation}
be an ordered basis of $V_\CC$ such that the matrix $[T_\CC]_\sB$ of $T_\CC$ relative to $\sB$ is a Jordan matrix
\begin{equation}\label{E:Jordan0}
[T_\CC]_\sB=\diag\big(J(\lambda_1,s_1),\ldots,J(\lambda_r,s_r)\big),
\end{equation}
where $J(\lambda_i,s_i)$ is the Jordan block with eigenvalue $\lambda_i$ and size $s_i$. Then $s=\max_{|\lambda_i|=\rho}s_i$.
By reordering the vectors in $\sB$, we may assume that there is $r_0\in\{1,\ldots,r\}$ such that
\begin{equation}\label{E:reord}
\begin{cases}
1\le i\le r_0 \qquad & \Longrightarrow \qquad  |\lambda_i|=\rho \text{ and } s_i=s,\\
r_0<i\le r \qquad & \Longrightarrow \qquad \text{either }  |\lambda_i|<\rho \text{ or } s_i<s.
\end{cases}
\end{equation}
Then, it is straightforward to verify that
\begin{equation}\label{E:pT}
[p(T)_\CC]_\sB=\diag(\overbrace{E_{1s},\ldots,E_{1s}}^{n\text{\ times}},0,\ldots,0),
\end{equation}
where $E_{1s}$ is the $s\times s$ matrix with $1$ in the $(1,s)$-entry and $0$ elsewhere. In turn, this implies that
\begin{align}
\ker\big(p(T)\big)_\CC & =\left\{\sum_{i=1}^r\sum_{j=1}^{s_i}x_{ij}\be_{ij}:x_{ij}\in\CC, x_{1s}=\cdots=x_{r_0,s}=0\right\}, \label{E:ker} \\
\im\big(p(T)\big)_\CC & =\left\{\sum_{i=1}^{r_0}x_{i1}\be_{i1}:x_{i1}\in\CC\right\}. \label{E:im}
\end{align}
It also follows from \eqref{E:pT} that if $T$ is $\RR$-diagonalizable, and if $V_\lambda\subset V$ is the eigenspace corresponding to $\lambda\in\Sp(T)$, then $$p(T)\text{ is the projection onto }\bigoplus_{\lambda\in\Sp(T),\,|\lambda|=\rho}V_\lambda \ \text{ along }\bigoplus_{\lambda\in\Sp(T),\,|\lambda|<\rho}V_\lambda.$$

Let us also observe the following fact.

\begin{lemma}\label{L:abelian}
Assume that $V$ is a Lie algebra, and $T$ is an automorphism of $V$ with $\rho=\rho(T)>1$. Then $\im\big(p(T)\big)$ is an abelian subalgebra of $V$.
\end{lemma}

\begin{proof}
It follows from \eqref{E:im} that the restriction of $T_\CC$ onto $\im\big(p(T)\big)_\CC$ is diagonalizable, and all eigenvalues of the restriction have modulus $\rho$. Therefore, it suffices to show that if $\lambda_1,\lambda_2\in\Sp(T)$, $|\lambda_1|=|\lambda_2|=\rho$, and $v_1,v_2\in V_\CC$ are such that $T_\CC v_i=\lambda_iv_i \ (i=1,2)$, then $[v_1,v_2]=0$. Suppose not. In view of
$$T_\CC([v_1,v_2])=[T_\CC v_1,T_\CC v_2]=\lambda_1\lambda_2[v_1,v_2],$$
it follows that $\lambda_1\lambda_2\in\Sp(T)$. But $|\lambda_1\lambda_2|=\rho^2>\rho$, a contradiction.
\end{proof}

\subsection{Expanding and maximally expanding horospherical subgroups}\label{affine}
Let $G$ be a Lie group with Lie algebra $\Lg$, $\Gamma\subset G$  a discrete subgroup, and $X=G/\Gamma$. Recall that a map $f:X\to X$ is affine if it is of the form  \equ{aff} for some   $g\in G$ and $\sigma\in\Aut(G,\Gamma)$. In this case, we also denote $f=f_{g,\sigma}$. Note that $f$ is always surjective, and is injective if and only if $\sigma(\Gamma)=\Gamma$. Let $\sigma_f$ be the automorphism of $G$ given by \equ{auto}, and let $d\sigma_f$ be the tangent map of $\sigma_f$ at $1_G$, which is an automorphism of $\Lg$. Let us observe the following simple facts.

\begin{lemma}\label{L:iterations}
Let $f=f_{g,\sigma}$ be an affine map on $X$.
\begin{itemize}
  \item[(1)] Let $g'\in G$ and $\sigma'\in\Aut(G,\Gamma)$ be such that $f_{g',\sigma'}=f$, and let $G^\circ$ be the identity component of $G$. Then there exists $\gamma\in\Gamma$ such that
\begin{equation}\label{E:affine1}
g'=g\gamma, \qquad \sigma'(h)=\gamma^{-1}\sigma(h)\gamma, \qquad \forall \, h\in G^\circ.
\end{equation}
  \item[(2)] The restriction of $\sigma_f$ to $G^\circ$, and hence $d\sigma_f$, is independent of the choices of $g$ and $\sigma$ that define $f$.
  \item[(3)] For every $n\ge0$, we have
  \begin{equation}\label{E:affine2}
     f^n(hx)=\sigma_f^n(h)f^n(x), \qquad \forall \, h\in G,\ x\in X.
  \end{equation}
\end{itemize}
\end{lemma}

\begin{proof}
(1) The condition implies that $\sigma(h)^{-1}g^{-1}g'\sigma'(h)\in\Gamma$ for all $h\in G$.
By taking $h=1_G$, we see that $g^{-1}g'=\gamma$ for some $\gamma\in\Gamma$. It follows that $\sigma(h)^{-1}\gamma\sigma'(h)\in\Gamma$ for all $h\in G$. Since $\Gamma$ is discrete, we have
$$\sigma(h)^{-1}\gamma\sigma'(h)=\sigma(1_G)^{-1}\gamma\sigma'(1_G)=\gamma, \qquad \forall \, h\in G^\circ.$$
This proves \eqref{E:affine1}.

\smallskip
(2) In view of (1), it suffices to verify that if $g,g',\sigma,\sigma',\gamma$ are such that \eqref{E:affine1} holds, then $g\sigma(h)g^{-1}=g'\sigma'(h)g'^{-1}$ for all $h\in G^\circ$. This is straightforward.

\smallskip
(3) If $n=0$, there is nothing to prove. For $n=1$, if $h\in G$ and $x=h'\Gamma\in X$, then
$$f(hx)=f(hh'\Gamma)=g\sigma(hh')\Gamma=(g\sigma(h)g^{-1})(g\sigma(h')\Gamma)=\sigma_f(h)f(x).$$
This shows that \eqref{E:affine2} holds for $n=1$. Assume $n\ge2$ and that \eqref{E:affine2} holds if $n$ is replaced by $1,\ldots,n-1$. Then for $h\in G$ and $x\in X$, we have
$$f^n(hx)=f^{n-1}\big(f(hx)\big)=f^{n-1}\big(\sigma_f(h)f(x)\big)=\sigma_f^{n-1}\big(\sigma_f(h)\big)f^{n-1}\big(f(x)\big)=\sigma_f^n(h)f^n(x).$$
This completes the proof.
\end{proof}

With the above lemma in mind, one can easily generalize the notion of the expanding   horospherical subgroup $G_f$ to the case when $f$ is an affine map:
the Lie algebra of $G_f$ is the subalgebra of $\Lg$ whose complexification is the
direct sum of generalized eigenspaces of $d\sigma_f$ corresponding to eigenvalues of modulus greater than $1$. Clearly it agrees with \equ{ehs} when $f\in G$.

Furthermore, let us now define the subgroup $G_f^{\max}$ mentioned in the introduction. Applying Lemma \ref{L:polynomial} to $V=\Lg$ and $T=d\sigma_f$, we get a polynomial $p(x)=p_{d\sigma_f}(x)$. If $\rho(d\sigma_f)\le1$, put $\Lg_f^{\max} = \{0\}$. Otherwise denote
\eq{maxexp}{\Lg_f^{\max}:=\im\big(p(d\sigma_f)\big),}
which is an abelian subalgebra of $\Lg$ by Lemma \ref{L:abelian}. After that we can define  $G_f^{\max}$ to be the connected Lie subgroup of $G$ with Lie algebra $\Lg_f^{\max}$.

From the preceding discussion it follows that another way of defining $\Lg_f^{\max}$ is as follows:
we can decompose the complexification  $\Lg_\CC$ of $\Lg$ as a direct sum $\bigoplus_{i=1}^r\Lg_i$ of $d\sigma_f$-invariant subspaces such that the matrix of the restriction of $d\sigma_f$ onto each $\Lg_i$, relative to a certain basis of $\Lg_i$, is a Jordan block with eigenvalue $\lambda_i$, then, reordering the $\Lg_i$'s, we may assume that there is $r_0\in\{1,\ldots,r\}$ such that $|\lambda_1|=\cdots=|\lambda_{r_0}|$, $\dim\Lg_1=\cdots=\dim\Lg_{r_0}$, and if $i>r_0$ then either $|\lambda_i|<|\lambda_1|$ or $|\lambda_i|=|\lambda_1|$ and $\dim\Lg_i<\dim\Lg_1$; finally,  we define $\Lg_f^{\max}$ as the intersection of $\Lg$ with the subspace of $\Lg_\CC$ spanned by the eigenvectors of $d\sigma_f$ contained in $\bigoplus_{i=1}^{r_0}\Lg_i$.
It follows from \equ{maxexp} that $\Lg_f^{\max}$ thus defined does not depend on the decomposition of $\Lg_\CC$. Note that if $d\sigma_f$ is diagonalizable over $\RR$, then $\Lg_f^{\max}$ is the sum of (real) eigenspaces corresponding to eigenvalues of $d\sigma_f$ with maximum modulus.

More generally, if $H$ is a closed subgroup of $G$ with Lie algebra $\Lh$, denote
$$\Lh_f^{\max}:=\begin{cases} \{0\} &\text{ if }\rho(d\sigma_f) \le 1;\\ p(d\sigma_f)(\Lh)&\text{ if }\rho(d\sigma_f) > 1.\end{cases}$$
Since $\Lh_f^{\max}\subset\Lg_f^{\max}$, it is also an abelian subalgebra. Let $H_f^{\max}$ denote the connected Lie subgroup of $G$ with Lie algebra $\Lh_f^{\max}$. Note that $H_f^{\max}$ is not necessarily contained in $H$.
\smallskip

Similarly one can define expanding and maximally expanding horospherical subgroups for one-parameter subsemigroups.
Let $G$ be as above, and let $F=\{g_t:t\in\RR\}$ be a one-parameter subgroup of $G$. If $f(x) := g_tx$; then we clearly have
$\sigma_f(h) = g_thg_{-t}$, hence $d\sigma_f = \Ad g_t$.
Then define  $$G_{F^\pm}:= G_{g_{\pm1}},\qquad  G_{F^\pm}^{\max}:= G_{g_{\pm1}}^{\max}.$$
Also if $H\subset G$ is a closed subgroup with Lie algebra $\Lh$, we will denote
$\Lh_{F^\pm}^{\max}:= \Lh_{g_{\pm1}}^{\max}$ and $H_{F^\pm}^{\max}:=H_{g_{\pm1}}^{\max}$.

\begin{example}\label{example0}
 Let $G = \SL_n(\RR)$, take $p,q\in\NN$ with $n = p + q$, and let
 \begin{equation}\label{E:gt-st}
F=\{g_t:t\in\RR\}, \quad \text{ where } \ g_t=\diag(e^{t/p}I_p,e^{-t/q}I_q),
\end{equation}
a subgroup of $G$ whose action on the quotient of $G$ by $ \SL_n(\ZZ)$ is useful for Diophantine applications, as we shall see in \S\ref{gibfs}.
Then both $\Ad g_1$  and $\Ad g_{-1}$ have  a unique eigenvalue of absolute value $> 1$, hence in this case there is no difference between expanding and   maximally expanding horospherical subgroups. Indeed, one has
 \begin{equation}\label{E:mainexample}\Lg_{F^+}^{\max}=\left\{\begin{pmatrix}
  0 & A \\ 0 & 0
\end{pmatrix}:A\in\mathrm{M}_{p\times q}(\RR)\right\}, \qquad \Lg_{F^-}^{\max}=\left\{\begin{pmatrix}
  0 & 0 \\ B & 0
\end{pmatrix}:B\in\mathrm{M}_{q\times p}(\RR)\right\}.\end{equation}
Note that any $F$ of the form \eqref{E:gt} is conjugate to \eqref{E:gt-st}.
\end{example}

\subsection{Nondense orbits of affine maps}\label{nondenseaffine}

Theorem \ref{A1} is a special case of the following theorem.

\begin{theorem}\label{T:main-discrete}
Let $G$ be a Lie group, $\Gamma\subset G$ a discrete subgroup, $X=G/\Gamma$, $H\subset G$ a closed subgroup, and $f$ an affine map on $X$. Let $Z$ be a $C^1$ submanifold of $X$ satisfying one of the following conditions:
\begin{itemize}
  \item[(i)] either
  \begin{equation}\label{E:discrete-main1}
  \dim\big(T_zZ\cap T_z(G_f^{\max}z)\big)<\dim H_f^{\max}  \qquad \forall \, z\in Z
  \end{equation}
  \item[(ii)] or
  $$\#\big\{\lambda\in\Sp(d\sigma_f): |\lambda|=\rho(d\sigma_f), \ s(\lambda)=s(d\sigma_f)\big\}=1$$
  and
  \begin{equation}\label{E:discrete-main2}
  T_z(H_f^{\max}z)\not\subset T_zZ  \qquad \forall \, z\in Z
  \end{equation}
(that is, $Z$ is $H_f^{\max}$-transversal).
\end{itemize}
Assume also that
\begin{equation}\label{E:discrete-main3}
T_z\big(\sigma_f^n(H)z\big)\not\subset T_zZ \qquad \forall \, z\in Z,\  n\ge0.
\end{equation}
Then, for every $x\in X$, the set
$\{h\in H: hx\in E(f,Z)\}$
is HAW.
\end{theorem}

\begin{remark}\label{R:main-discrete}
\begin{itemize}
  \item[(1)] If $\rho(d\sigma_f) \le 1$, the group  $H_f^{\max}$ is trivial, thus neither \eqref{E:discrete-main1} nor \eqref{E:discrete-main2} can be satisfied. Thus without loss of generality one can assume that  $\rho(d\sigma_f) > 1$.
\smallskip

\item[(2)] When $H=G$, both \eqref{E:discrete-main1} and \eqref{E:discrete-main2} are equivalent to the condition that $Z$ is $G_f^{\max}$-transversal, and \eqref{E:discrete-main3} is always satisfied as long as $\dim Z < \dim X$.  Therefore Theorem \ref{T:main-discrete} implies Theorem \ref{A1}.
\smallskip

\item[(3)] If $Z$ is a point, then \eqref{E:discrete-main3} always holds, and both \eqref{E:discrete-main1} and \eqref{E:discrete-main2} are equivalent to $\dim H_f^{\max}>0$, which happens if and only if $\rho(d\sigma_f) > 1$ and $\Lh\not\subset\ker\big(p(d\sigma_f)\big)$.
\smallskip

\item[(4)]  One can also take $H$ to be a one-parameter subgroup. In this case, condition \eqref{E:discrete-main1} is stronger than \eqref{E:discrete-main2}.
\smallskip

\item[(5)] Condition \eqref{E:discrete-main3} is imposed to exclude the case where $Z$ contains an open subset of $f^n(Hx)$ for some $x\in X$ and $n\ge0$. If condition \eqref{E:discrete-main3} is dropped, it can be shown that for every $x\in X$, the set $\{h\in H:\omega(hx)\cap Z=\emptyset\}$, where
    $$\omega(hx):=\{y\in X:\exists \ {n_k\to+\infty \text{ such that } f^{n_k}}hx\to y\}$$
   is the $\omega$-limit set of $hx$,  is HAW on $H$. See Remark \ref{R:omega2}.
\end{itemize}
\end{remark}

\subsection{Nondense orbits of continuous flows}\label{sub:continuous}
\ignore{
Let $G, \Gamma$ and $X$ be as above, and let $F=\{g_t:t\in\RR\}$ be a one-parameter subgroup of $G$. Put $f(x) := g_1x$; then we clearly have $\sigma_f(h) = g_1hg_1^{-1}$, hence $d\sigma_f = \Ad g_1$. Recall that we denoted the group $G_{g_1}^{\max}$ by $G_{F^+}^{\max}$. Also if $H\subset G$ is a closed subgroup with Lie algebra $\Lh$, denote
$\Lh_{F^+}^{\max}:= \Lh_{g_1}^{\max}$ and $H_{F^+}^{\max}:=H_{g_1}^{\max}$.
\smallskip
}
We are now ready to state a continuous analogue of Theorem \ref{T:main-discrete}.

\begin{theorem}\label{T:main-continuous}
Let $G$ be a Lie group, $\Gamma\subset G$ a discrete subgroup, $X=G/\Gamma$, and $H\subset G$ a closed subgroup. Let $F=\{g_t:t\in\RR\}$ be a one-parameter subgroup of $G$. Let $Z$ be an $F$-transversal $C^1$ submanifold of $X$. Assume that
\begin{itemize}
  \item[(i)] either
\begin{equation}\label{E:continuous-main2}
\dim\big((T_zZ\oplus T_z(Fz)\big)\cap T_z\big(G_{F^+}^{\max}z)\big)<\dim H_{F^+}^{\max} \qquad \forall \, z\in Z;
\end{equation}
  \item[(ii)] or $F$ is $\Ad$-diagonalizable over $\RR$ and
\begin{equation}\label{E:continuous-main3}
T_z(H_{F^+}^{\max}z)\not\subset T_zZ\oplus T_z(Fz) \qquad \forall \, z\in Z
\end{equation}
(that is, $Z$ is $(F,H_{F^+}^{\max})$-transversal).
\end{itemize}
Assume also that
\begin{equation}\label{E:continuous-main4}
T_z(g_tHg_t^{-1}z)\not\subset T_zZ\oplus T_z(Fz) \qquad \forall \, z\in Z,\ t\ge0.
\end{equation}
Then, for every $x\in X$ the set
\begin{equation}\label{E:continuous-main5}
\{h\in H: hx\in E(F^+,Z)\}
\end{equation}
is HAW.
\end{theorem}

\begin{remark}\label{R:main-continuous}
\begin{itemize}
\item[(1)] Similarly to Theorem \ref{T:main-discrete},  neither \eqref{E:continuous-main2} nor \eqref{E:continuous-main3} can hold if $\rho(\Ad g_1) \le 1$. Thus without loss of generality one can assume that  $\rho(\Ad g_1) > 1$.
\smallskip

\item[(2)] As in the case of Theorem \ref{T:main-discrete}, the $H=G$ case of Theorem \ref{T:main-continuous} implies Theorem \ref{A2}. In fact, in this situation, in view of the assumption of $F$-transversality of $Z$, both \eqref{E:continuous-main2} and \eqref{E:continuous-main3} are equivalent to the condition that $Z$ is $(F,G_{F^+}^{\max})$-transversal, and \eqref{E:continuous-main4} is always satisfied.
\smallskip

\item[(3)] Assume $Z$ is a point. Then it is $F$-transversal. Since the intersection of $\Lf$ and $\Lg_{F^+}^{\max}$ is always trivial, both \eqref{E:continuous-main2} and \eqref{E:continuous-main3} are equivalent to $\dim H_{F^+}^{\max}>0$, which happens if and only if $\rho(\Ad g_1) > 1$ and $\Lh\not\subset\ker\big(p(\Ad g_1)\big)$. Note also that \eqref{E:continuous-main4} means $\Lh\not\subset\Lf$, which automatically holds if \eqref{E:continuous-main2} or \eqref{E:continuous-main3} is satisfied.
\smallskip

\item[(4)] Condition \eqref{E:continuous-main4} is imposed to exclude the case where $F^-Z$ contains an open subset of $Hx$ for some $x\in X$.
    If condition \eqref{E:continuous-main4} is dropped, it can be shown that for every $x\in X$, the set $\{h\in H:\omega(hx)\cap Z=\emptyset\}$ is HAW on $H$, where
    $$\omega(hx):=\left\{y\in X:\exists \ t_k\to+\infty, g_{t_k}hx\to y\right\}$$
   is the $\omega$-limit set of $hx$. See Remark \ref{R:omega2}.
\end{itemize}
\end{remark}

\subsection{Proof of Theorem \ref{T:main-continuous} from Theorem \ref{T:main-discrete}}

We now deduce Theorem \ref{T:main-continuous} from Theorem \ref{T:main-discrete}. Assume the conditions of Theorem \ref{T:main-continuous} hold.
Since any $C^1$ submanifold of $X$ is the union of countably many compact $C^1$ submanifolds (possibly with boundaries), we may assume without loss of generality that $Z$ is compact. In this case, it follows from the $F$-transversality of $Z$ that the set
$$Z_{[0,\tau]}:=\bigcup_{t\in[0,\tau]}g_tZ$$ is a $C^1$ submanifold of $X$ for some $\tau>0$, and we have $T_zZ_{[0,\tau]}=T_zZ\oplus T_z(Fz)$ for every $z\in Z$. Moreover, shrinking $\tau$ if necessary, condition \eqref{E:continuous-main2} (resp.\ \eqref{E:continuous-main3}, \eqref{E:continuous-main4}) implies that \eqref{E:discrete-main1} (resp.\ \eqref{E:discrete-main2}, \eqref{E:discrete-main3}) holds with $f(x)=g_\tau x$ and $Z$ replaced by $Z_{[0,\tau]}$ (see \cite[Lemma 4.1.2]{Kl2} or \cite[Lemma 4.1]{KW3} for details). Note also that $p_{\Ad g_\tau}(\Ad g_\tau)=p_{\Ad g_1 } (\Ad g_1 )$. Therefore, Theorem \ref{T:main-discrete} implies that the set
\begin{equation}\label{E:EgZ}
\left\{h\in H:\overline{\{g_{n\tau}hx:n\ge0\}}\cap Z_{[0,\tau]}=\emptyset\right\}
\end{equation}
is HAW on $H$. On the other hand, the set \eqref{E:EgZ} is contained in the set \eqref{E:continuous-main5}. In fact, if $h\in H$ is not in \eqref{E:continuous-main5}, then there exist $t_k\ge0$ such that $g_{t_k}hx\to z\in Z$. Let $n_k\ge0$ be such that $n_k\tau-t_k\in[0,\tau)$. By passing to a subsequence, we may assume that $n_k\tau-t_k\to t\in[0,\tau]$. It follows that
$$g_{n_k\tau}hx=g_{n_k\tau-t_k}(g_{t_k}hx)\to g_tz\in Z_{[0,\tau]}.$$
Thus $h$ is not in \eqref{E:EgZ}. This shows that the set \eqref{E:continuous-main5} contains \eqref{E:EgZ}, and hence is HAW on $H$. \qed

\section{Hyperplanes in a subspace}\label{hyp}

Let $V$ be a Euclidean space with inner product $\la\cdot,\cdot\ra$, and let $\sL(V)$ denote the vector space of linear transformations on $V$. Both the Euclidean norm on $V$ and the operator norm on $\sL(V)$ are denoted by $\|\cdot\|$. For $0\le d\le \dim V$, let $\Gr_d(V)$ denote the Grassmann manifold of $d$-dimensional subspaces of $V$. Our primary goal in this section is to prove the following result concerning hyperplanes in a subspace $U$ of $V$.

\begin{proposition}\label{P:linear-main}
Let $T\in\GL(V)$, $U$ a nonzero subspace of $V$, $0\le d\le \dim V-1$, $\sW$ a closed subset of $\Gr_d(V)$, and $p(x)$ the polynomial given by Lemma \ref{L:polynomial}. Suppose that
\begin{itemize}
  \item[(i)]  either $$\dim\big(W\cap\im p(T)\big)<\dim \big(p(T)U\big) \qquad \forall \, W\in\sW,$$
  \item[(ii)]  or
  $$\#\{\lambda\in\Sp(T): |\lambda|=\rho(T), \ s(\lambda)=s(T)\}=1,$$
  and $p(T)U\not\subset W$ for every $W\in\sW$.
\end{itemize}
Suppose also that
\begin{equation}\label{E:linear-main}
T^n(U)\not\subset W, \qquad \forall \, W\in\sW, n\ge0.
\end{equation}
Then there exists a constant $c=c(T,U,\sW)>0$ satisfying the following property: For any $W\in\sW$ and $n\ge0$, there exists a linear hyperplane $L_{W,n}$ in $U$ such that
\begin{equation}\label{E:linear-main2}
\dist(T^n\bu,W)\ge c\|T^n\|\dist(\bu,L_{W,n}), \qquad \forall \,  \bu\in U.
\end{equation}
\end{proposition}

Let us remark that if $W^{(1)}$ denotes the $1$-neighborhood of $W$ in $V$, inequality \eqref{E:linear-main2} means that $T^{-n}(W^{(1)})\cap U$ is contained in the $(c\|T^n\|)^{-1}$-neighborhood of $L_{W,n}$ in $U$.
\smallskip

We first prove some auxiliary lemmas. The first one is probably well known, but we could not find an appropriate reference. We give its simple proof for completeness.

\begin{lemma}\label{L:linear1}
Suppose that $T\in\sL(V)$ is not nilpotent, and let $\rho=\rho(T)$, $s=s(T)$. Then there exists $C>1$ such that
$$C^{-1}n^{s-1}\rho^n\le\|T^n\|\le Cn^{s-1}\rho^n, \qquad \forall \,  n\ge0.$$
\end{lemma}

\begin{proof}
By replacing $T$ with $T/\rho$, we may assume that $\rho=1$. Let $\B$ be an ordered basis of $V_\CC$ such that the matrix $[T_\CC]_\sB$ is the Jordan normal form \eqref{E:Jordan0}, and let $\|\cdot\|_\sB$ be the norm on $\sL(V)$ given by $\|S\|_\sB=\|[S_\CC]_\sB\|_\infty$, where $\|\cdot\|_\infty$ denotes the largest modulus of the matrix entries. Then
$$\|T^n\|_\sB=\|\diag(J(\lambda_1,s_1)^n,\ldots,J(\lambda_r,s_r)^n)\|_\infty=\max_{1\le i\le r}\|J(\lambda_i,s_i)^n\|_\infty.$$
It is straightforward to verify that for $1\le j\le k\le s_i$,
\begin{equation}\label{E:Jordan-power}
\text{the $(j,k)$-entry of $J(\lambda_i,s_i)^n$ is equal to $\bin{n}{k-j}\lambda_i^{n-(k-j)}$.}
\end{equation}
Since $|\lambda_i|\le1$, this implies that $\|J(\lambda_i,s_i)^n\|_\infty=\bin{n}{s_i-1}|\lambda_i|^{n-s_i+1}$ whenever $n\ge2s_i$.
Thus, there exists $n_0>0$ such that
$$\|T^n\|_\sB=\max_{1\le i\le r}\bin{n}{s_i-1}|\lambda_i|^{n-(s_i-1)}=\bin{n}{s-1}, \qquad \forall \,  n\ge n_0.$$
Now the lemma follows from the fact that any two norms on $\sL(V)$ are equivalent.
\end{proof}

Let $\SS(V)$ be the unit sphere in $V$, that is,
$$\SS(V)=\{{\bv}\in V:\|{\bv}\|=1\}.$$
Then every $T\in\sL(V)$ induces a map
$$\la T\ra:\SS(V)\sm\ker T \to\SS(V), \qquad {\bv}\mt\frac{T{\bv}}{\|T{\bv}\|}.$$
The next lemma explains the role of the polynomial $p(x)$.

\begin{lemma}\label{L:linear2}
Let $T\in\GL(V)$, let $p(x)$ be the polynomial given by Lemma \ref{L:polynomial}, and let $K$ be a compact subset of $\SS(V)\sm\ker  p(T) $. Then
\begin{equation}\label{E:T}
\inf_{{\bv}\in K,\,n\ge0}\frac{\|T^n{\bv}\|}{\|T^n\|}>0
\end{equation}
and
\begin{equation}\label{E:p}
\lim_{n\to+\infty}\sup_{{\bv}\in K}\|\la T^n\ra {\bv}-\la T^{n-s+1}p(T)\ra {\bv}\|=0.
\end{equation}
\end{lemma}

\begin{proof}
As in the proof of Lemma \ref{L:linear1}, we may assume that $\rho(T)=1$. Let $\B$ be an ordered basis of $V_\CC$ of the form \eqref{E:basis} such that $[T_\CC]_\sB$ is the matrix \eqref{E:Jordan0} and satisfies \eqref{E:reord}. In this proof, we always write a vector ${\bv}\in V$ as
$${\bv}=\sum_{i=1}^r\sum_{j=1}^{s_i}x_{ij}\be_{ij}.$$
It then follows from \eqref{E:Jordan-power} that
\begin{equation}\label{E:Tnv}
T^n{\bv}=\sum_{i=1}^r\sum_{j=1}^{s_i}\left(\sum_{k=j}^{s_i}\bin{n}{k-j}\lambda_i^{n-(k-j)}x_{ik}\right)\be_{ij}.
\end{equation}
Let $c_0>0$ be such that for any ${\bv}\in V$ we have
\begin{equation}\label{E:c11}
\|{\bv}\|\ge c_0\max_{1\le i\le r_0}|x_{i1}|.
\end{equation}
In view of \eqref{E:ker} and the conditions on $K$, we may also assume that if ${\bv}\in K$ then
\begin{equation}\label{E:c12}
\max_{1\le i\le r_0}|x_{is}|\ge c_0 \qquad \text{and} \qquad \max_{1\le i\le r,1\le j\le s_i}|x_{ij}|\le c_0^{-1}.
\end{equation}
It follows that for $n>n_1:=\lfloor2s^2(1+c_0^{-2})\rfloor$ and ${\bv}\in K$, we have
\begin{eqnarray}
\bin{n}{s-1}^{-1}\|T^n{\bv}\|
& \stackrel{\eqref{E:Tnv},\,\eqref{E:c11}}\ge & c_0\max_{1\le i\le r_0}\left|\sum_{k=1}^s\bin{n}{s-1}^{-1}\bin{n}{k-1}\lambda_i^{n-(k-1)}x_{ik}\right| \notag \\
& \ge & c_0\max_{1\le i\le r_0}\left(|x_{is}|-\sum_{k=1}^{s-1}\bin{n}{s-1}^{-1}\bin{n}{k-1}|x_{ik}|\right) \notag \\
& \stackrel{\eqref{E:c12}}\ge & c_0\left(c_0-(s-1)\bin{n}{s-1}^{-1}\bin{n}{s-2}c_0^{-1}\right) \notag \\
& = & c_0^2-\frac{(s-1)^2}{n-s+2}\ge \frac{c_0^2}{2}. \label{E:big}
\end{eqnarray}
This, together with Lemma \ref{L:linear1}, shows that $\inf_{{\bv}\in K,n>n_1}\frac{\|T^n{\bv}\|}{\|T^n\|}>0$. Clearly, we also have $\inf_{{\bv}\in K,\,0\le n\le n_1}\frac{\|T^n\bv\|}{\|T^n\|}>0$. This proves \eqref{E:T}.
\smallskip

We now prove \eqref{E:p}. For $n>n_1$ and ${\bv}\in K$, we have
\begin{eqnarray*}
& & \ \|\la T^n\ra {\bv}-\la T^{n-s+1}p(T)\ra {\bv}\|\\
& \le & \ \left\|\frac{T^n{\bv}}{\|T^n{\bv}\|}-\frac{\bin{n}{s-1}T^{n-s+1}p(T){\bv}}{\|T^n{\bv}\|}\right\|
+\left\|\frac{\bin{n}{s-1}T^{n-s+1}p(T){\bv}}{\|T^n{\bv}\|}-\frac{T^{n-s+1}p(T){\bv}}{\|T^{n-s+1}p(T){\bv}\|}\right\|\\
& = & \ \frac{\|T^n{\bv}-\bin{n}{s-1}T^{n-s+1}p(T){\bv}\|}{\|T^n{\bv}\|}
+\left|\frac{\|T^n{\bv}\|-\bin{n}{s-1}\|T^{n-s+1}p(T){\bv}\|}{\|T^n{\bv}\|}\right|\\
& \stackrel{\eqref{E:big}}\le & \ 4c_0^{-2}\left\|\bin{n}{s-1}^{-1}T^n{\bv}-T^{n-s+1}p(T){\bv}\right\|.
\end{eqnarray*}
Let us write
$${\bin{n}{s-1}}^{-1}T^n{\bv}-T^{n-s+1}p(T){\bv}=\sum_{i=1}^r\sum_{j=1}^{s_i}f_{ij}^{(n)}({\bv})\be_{ij}.$$
It suffices to prove that
\begin{equation}\label{E:uniform}
\lim_{n\to+\infty}\sup_{{\bv}\in K}|f_{ij}^{(n)}({\bv})|=0
\end{equation}
for any $1\le i\le r$ and $1\le j\le s_i$.
It follows from \eqref{E:pT} and \eqref{E:Tnv} that
$${\bin{n}{s-1}}^{-1}T^n{\bv}-T^{n-s+1}p(T){\bv}
=\sum_{i=1}^r\sum_{j=1}^{s_i}\left(\sum_{k=j}^{s_i}{\bin{n}{s-1}}^{-1}\bin{n}{k-j}\lambda_i^{n-(k-j)}x_{ik}\right)\be_{ij}
-\sum_{i=1}^{r_0}\lambda_i^{n-(s-1)}x_{is}\be_{i1}.$$
Thus, by comparing the coefficients, we deduce that for $1\le i\le r_0$, $j=1$, $n\ge2s$ and ${\bv}\in K$,
\begin{align*}
|f_{i1}^{(n)}({\bv})| & = \left|\left(\sum_{k=1}^s\bin{n}{s-1}^{-1}\bin{n}{k-1}\lambda_i^{n-(k-1)}x_{ik}\right)-\lambda_i^{n-(s-1)}x_{is}\right|\\
& =\left|\sum_{k=1}^{s-1}\bin{n}{s-1}^{-1}\bin{n}{k-1}\lambda_i^{n-(k-1)}x_{ik}\right|\le\frac{(s-1)^2}{c_0(n-s+2)}.
\end{align*}
Hence \eqref{E:uniform} holds for $1\le i\le r_0$ and $j=1$. Similarly, for $1\le i\le r_0$, $2\le j\le s$, $n\ge2s$ and ${\bv}\in K$, we have
$$|f_{ij}^{(n)}({\bv})|=\left|\sum_{k=j}^s\bin{n}{s-1}^{-1}\bin{n}{k-j}\lambda_i^{n-(k-j)}x_{ik}\right|\le\frac{(s-j+1)(s-1)}{c_0(n-s+2)},$$
and hence \eqref{E:uniform} holds for $1\le i\le r_0$ and $2\le j\le s$.
Finally, for $r_0<i\le r$, $1\le j\le s_i$, $n\ge2s_i$ and ${\bv}\in K$, we have
$$ |f_{ij}^{(n)}({\bv})|=\left|\sum_{k=j}^{s_i}\bin{n}{s-1}^{-1}\bin{n}{k-j}\lambda_i^{n-(k-j)}x_{ik}\right|\le s_i\bin{n}{s_i-1}|\lambda_i|^{n-(s_i-1)}c_0^{-1}.$$
Since $|\lambda_i|<1$, we also have \eqref{E:uniform} for $i$ and $j$ in these ranges. This completes the proof.
\end{proof}

We will also need the following result. 

\begin{lemma}\label{L:linear3}
Let $T$, $U$, $\sW$ and $p(x)$ be as in Proposition \ref{P:linear-main}, let $T^*$ be the adjoint transformation of $T$, and suppose that one of the conditions {\rm (i)} or {\rm (ii)} in the Proposition holds. Then there exist compact subsets
$$K\subset\SS(V)\sm\ker p(T^*) \qquad \text{and} \qquad K^{(1)}\subset\SS(V)\sm U^\bot$$
such that
\begin{equation}\label{E:linear3}
\la(T^*)^{n-s+1}p(T^*)\ra(K\cap W^\bot)\cap K^{(1)}\ne\emptyset  \qquad \forall \, W\in\sW, \ n\ge0.
\end{equation}
\end{lemma}

\begin{proof}
(1) Assume that condition (i) holds. Let $V^{(1)}\subset\im p(T^*) $ be a subspace such that
$$\im p(T^*) =\big(U^\bot\cap\im p(T^*) \big)\oplus V^{(1)},$$
and let $K^{(1)}=\SS(V^{(1)})$. Then $K^{(1)}\subset\SS(V)\sm U^\bot$.
We first prove that for every $W_0\in\sW$, there exist a neighborhood $\sN_{W_0}$ of $W_0$ in $\sW$ and a compact subset $K_{W_0}$ of \linebreak $\SS(V)\sm\ker p(T^*)$ such that
\begin{equation}\label{E:local}
\la(T^*)^{n-s+1}p(T^*)\ra(K_{W_0}\cap W^\bot)\cap K^{(1)}\ne\emptyset  \qquad \forall \, W\in\sN_{W_0}, \ n\ge0.
\end{equation}
To show this, let $V_0\subset W_0^\bot$ be a subspace such that $$W_0^\bot=\big(\ker\, p(T^*)\cap W_0^\bot\big)\oplus V_0,$$
and choose a neighborhood $\sN_{W_0}$ of $W_0$ in $\sW$ and a continuous map
$\overline{\sN_{W_0}}\to\Gr_{\dim V_0}(V)$, $W\mt V_W$ with $V_{W_0}=V_0$ such that $V_W\subset W^\bot$ for $W\in\overline{\sN_{W_0}}$. Shrinking $\sN_{W_0}$ if necessary, we may also assume that $V_W\cap\ker p(T^*) =\{0\}$ for $W\in\overline{\sN_{W_0}}$. Then the set $$K_{W_0}=\bigcup_{W\in\overline{\sN_{W_0}}}\SS(V_W)$$
is compact and contained in $\SS(V)\sm\ker p(T^*)$. For $W\in\sN_{W_0}$ and $n\ge0$, we have
\begin{align*}
\dim\big((T^*)^{n-s+1}p(T^*)V_W\cap V^{(1)}\big) & \ge\dim \, (T^*)^{n-s+1}p(T^*)V_W +\dim V^{(1)}-\dim \im p(T^*) \\
& =\dim   V_0-\big(\dim\im p(T^*)-\dim V^{(1)}\big)\\
& =\dim p(T^*)W_0^\bot -\dim\big(U^\bot\cap\im  p(T^*)\big)\\
& =\dim  \big(W_0\cap\im p(T)\big)^\bot-\dim \big(p(T)U\big)^\bot>0.
\end{align*}
Thus $(T^*)^{n-s+1}p(T^*)(V_W)\cap V^{(1)}\ne\{0\}$. Therefore, there exists ${\bv}\in\SS(V_W)\subset K_{W_0}\cap W^\bot$ such that $\la(T^*)^{n-s+1}p(T^*)\ra {\bv}\in K^{(1)}$. This proves that $\sN_{W_0}$ and $K_{W_0}$ satisfy \eqref{E:local}.

For every $W_0\in\sW$, let us choose $\sN_{W_0}$ and $K_{W_0}$ satisfying \eqref{E:local}. Since $\sW$ is compact, there exist $W_1,\ldots,W_m\in\sW$ such that $\sW=\bigcup_{i=1}^m\sN_{W_i}$. Then $K=\bigcup_{i=1}^mK_{W_i}$ satisfy the requirement of the lemma.
\smallskip

(2) We now assume that condition (ii) holds. We first construct a compact subset $K\subset\SS(V)\sm\ker p(T^*)$ such that $-K=K$, $p(T^*)K\cap U^\bot=\emptyset$ and $K\cap W^\bot\ne\emptyset$ for every $W\in\sW$. Let $W_0\in\sW$. It follows from condition (ii) that $p(T^*)W_0^\bot \not\subset U^\bot$. Let ${\bv}_0\in\SS(W_0^\bot)$ be such that $p(T^*){\bv}_0\notin U^\bot$. Then we can choose a neighborhood $\sN_{W_0}$ of $W_0$ in $\sW$ and a continuous map $\overline{\sN_{W_0}}\to\SS(V)$, $W\mt {\bv}_{W}$ such that ${\bv}_{W_0}={\bv}_0$, ${\bv}_{W}\in W^\bot$ and $p(T^*){\bv}_W\notin U^\bot$ for $W\in\overline{\sN_{W_0}}$. The compact set $K_{W_0}=\{v_W:W\in\overline{\sN_{W_0}}\}$ satisfies $K_{W_0}\cap\ker p(T^*) =\emptyset$, $p(T^*)K_{W_0}\cap U^\bot=\emptyset$ and $K_{W_0}\cap W^\bot\ne\emptyset$ for every $W\in\sN_{W_0}$.
Let $W_1,\ldots,W_m\in\sW$ be such that $\sW=\bigcup_{i=1}^m\sN_{W_i}$. Then the set $K=\bigcup_{i=1}^m\big(K_{W_i}\cup(-K_{W_i})\big)$ satisfies the requirement.

Let $K^{(1)}=\la p(T^*)\ra K\subset\SS(V)\sm U^\bot$. Since $T^*$ and $T$ have the same minimal polynomial, we have $p_{T^*}(x)=p(x)$.
It then follows from condition (ii) and Lemma \ref{L:polynomial} that the restriction of $\la T^*\ra$ to $\SS\big(\im p(T^*)\big)$ is $\pm1$. Therefore, for every $W\in\sW$ and $n\ge0$, we have
$$\la(T^*)^{n-s+1}p(T^*)\ra(K\cap W^\bot)\cap K^{(1)}=\la p(T^*)\ra(K\cap W^\bot)\ne\emptyset.$$
This completes the proof.
\end{proof}

We are now prepared to prove the main result of this section.

\begin{proof}[Proof of Proposition \ref{P:linear-main}]
Let $K$ and $K^{(1)}$ be the compact sets given by Lemma \ref{L:linear3}, and let $K^{(2)}$ be a compact neighborhood of $K^{(1)}$ in $\SS(V)\sm U^\bot$. Since $p_{T^*}(x)=p(x)$, applying Lemma \ref{L:linear2} to $T^*$, we get
$$\lim_{n\to+\infty}\sup_{{\bv}\in K}\left\|\la(T^*)^n\ra {\bv}-\la(T^*)^{n-s+1}p(T^*)\ra {\bv}\right\|=0.$$
Therefore, there exists $N\ge0$ such that for $n\ge N$ and ${\bv}\in K$, we have
$$\la(T^*)^{n-s+1}p(T^*)\ra {\bv}\in K^{(1)} \qquad \Longrightarrow \qquad \la(T^*)^n\ra {\bv}\in K^{(2)}.$$
For $n\ge N$ and $W\in\sW$, it follows from \eqref{E:linear3} that we can choose ${\bv}_{W,n}\in K\cap W^\bot$ such that $\la(T^*)^{n-s+1}p(T^*)\ra {\bv}_{W,n}\in K^{(1)}$, and hence $\la(T^*)^n\ra {\bv}_{W,n}\in K^{(2)}$.
For $0\le n<N$, using condition \eqref{E:linear-main}  and arguing as part (2) of the proof of Lemma \ref{L:linear3} (with $p(T^*)$ replaced by $(T^*)^n$), we see that there exists a compact subset $K_n\subset\SS(V)$ such that $(T^*)^n(K_n)\cap U^\bot=\emptyset$ and $K_n\cap W^\bot\ne\emptyset$ for every $W\in\sW$. In this case, we choose ${\bv}_{W,n}\in K_n\cap W^\bot$. Let
$$K^{(3)}=K^{(2)}\cup\bigcup_{0\le n<N}\la(T^*)^n\ra(K_n),$$
which is again a compact subset of $\SS(V)\sm U^\bot$. In summary, for every $n\ge0$ and $W\in\sW$, we have chosen a unit vector ${\bv}_{W,n}\in W^\bot$ with $\la(T^*)^n\ra {\bv}_{W,n}\in K^{(3)}$. When $n\ge N$, we also have ${\bv}_{W,n}\in K$.
\smallskip

Let $P_U\in\sL(V)$ be the orthogonal projection onto $U$, and let
$$c_1=\inf_{{\bv}\in K^{(3)}}\|P_U\bv\|>0,$$
$$c_2=\min\left\{\inf_{{\bv}\in K,\,n\ge N}\frac{\|(T^*)^n{\bv}\|}{\|T^n\|},\inf_{{\bv}\in\SS(V),\,0\le n<N}\frac{\|(T^*)^n{\bv}\|}{\|T^n\|}\right\}\stackrel{\eqref{E:T}}>0.$$
Then for any $n\ge0$ and $W\in\sW$, we have
$$\|P_U(T^*)^n{\bv}_{W,n}\|\ge c_1\|(T^*)^n{\bv}_{W,n}\|\ge c_1c_2\|T^n\|.$$
Since $(T^*)^n{\bv}_{W,n}\notin U^\bot$, the intersection
$$L_{W,n}=((T^*)^n{\bv}_{W,n})^\bot\cap U$$
is a hyperplane in $U$. For ${\bu}\in U$, we have
$$\dist({\bu},L_{W,n})=\frac{|\la {\bu},P_U(T^*)^n{\bv}_{W,n}\ra|}{\|P_U(T^*)^n{\bv}_{W,n}\|}
=\frac{|\la T^n{\bu},{\bv}_{W,n}\ra|}{\|P_U(T^*)^n{\bv}_{W,n}\|}
\le\frac{\dist(T^n{\bu},W)}{c_1c_2\|T^n\|}.$$
This completes the proof.
\end{proof}

\begin{remark}\label{R:omega1}
In the proof of Proposition \ref{P:linear-main}, condition \eqref{E:linear-main} is only used to define the sets $K_n$ for $0\le n<N$. If condition \eqref{E:linear-main} is dropped, the same argument (for $K^{(3)}=K^{(2)}$ and $n\ge N$) shows the following weaker statement: There exist $N>0$ and $c>0$ such that for $W\in\sW$ and $n\ge N$, there exists a linear hyperplane $L_{W,n}$ in $U$ such that \eqref{E:linear-main2} holds.
\end{remark}

\section{Proof of Theorem \ref{T:main-discrete}}\label{proofs}

\subsection{Hyperplane percentage game}\label{S:HPW}

We will prove the HAW property by demonstrating the winning property for the hyperplane percentage game introduced in \cite{BFS}. Being played on an open subset $U$ of a Euclidean space $V$, the hyperplane percentage game has the same winning sets as the hyperplane absolute game.

Let $S\subset U$ be a target set, and let $\beta\in(0,1)$. The \textsl{$\beta$-hyperplane percentage game} is defined as follows: Bob begins by choosing a closed Euclidean ball $B_0\subset U$. After Bob chooses a closed ball $B_i$ of radius ${r}_i$, Alice chooses finitely many hyperplane neighborhoods $\{L_{i,j}^{({r}_{i,j})} : 1\le j\le N_i\}$ such that $r_{i,j}\le\beta r_i$, and then Bob chooses a closed ball $B_{i+1}\subset B_i$ of radius $r_{i+1}\ge\beta r_i$ such that
$$\#\{1\le j\le N_i:B_{i+1}\cap L_{i,j}^{({r}_{i,j})}=\emptyset\}\ge N_i/2.$$
Alice wins the game if
$$\bigcap_{i=0}^\infty B_i\cap S\ne\emptyset.$$
The set $S$ is \textsl{$\beta$-hyperplane percentage winning ($\beta$-HPW) on $U$} if Alice has a winning
strategy. Note that for large values of $\beta$, it is possible for Alice to leave Bob with no available moves after finitely many turns. However, an elementary argument (see \cite[Lemma 2]{Mo} or \cite[\S2]{BFK}) shows that Bob always has a legal move if $\beta$ is smaller than some constant $\beta_0(\dim V)<1$. For example, we have $\beta_0(1)=1/5$. The set $S$ is \textsl{hyperplane percentage winning (HPW) on $U$} if it is $\beta$-HPW on $U$ for any $\beta\in\big(0,\beta_0(\dim V)\big)$. The significance of this notion lies in the following result.

\begin{lemma}[\cite{BFS}]\label{L:HPW2}
Let $U$ be an open subset of a Euclidean space $V$. A subset $S\subset U$ is HPW on $U$ if and only if it is HAW on $U$.
\end{lemma}

Let us remark that when proving a set $S$ to be HPW, we may assume that $r_i\to0$. In fact, if Alice has a winning strategy whenever $r_i\to0$, then $S$ must be dense, and hence Alice always wins if $r_i\not\to0$. Moreover, by letting Alice make dummy moves in the first several rounds and relabeling $B_i$, we may also assume that $r_0$ is smaller than any prescribed small positive constant.

\subsection{Some Lie-theoretic lemmas}\label{lie}

Let $G$ be a Lie group with Lie algebra $\Lg$. We choose and fix an inner product on $\Lg$. For an inner product space $V$ and $\tau>0$, let $B_V(\tau)$ (resp.\ $B_V^\circ(\tau)$) denote the closed ball (resp.\ open ball) in $V$ of radius $\tau$ centered at $0$. Let $\tau_1>0$ be such that the exponential map of $G$ restricts to a diffeomorphism from $B_\Lg^\circ(\tau_1)$ onto an open neighborhood of $1_G$ in $G$, and let
$$\log:\exp\big(B_\Lg^\circ(\tau_1)\big)\to B_\Lg^\circ(\tau_1)$$
be the inverse of $\exp|_{B_\Lg^\circ(\tau_1)}$. Let $\tau_2\in(0,\tau_1]$ be such that
$$\bx_1,\bx_2,\bx_3\in B_\Lg(\tau_2) \qquad \Longrightarrow \qquad \exp(\bx_1)\exp(\bx_2)\exp(\bx_3)\in\exp\big(B_\Lg^\circ(\tau_1)\big).$$
First, let us prove the following lemma.

\begin{lemma}\label{L:Lie0}
For any $\varepsilon>0$, there exists $\tau_3=\tau_3(\varepsilon)\in(0,\tau_2]$ such that if $\bx,\by,\bz\in B_\Lg(\tau_3)$ satisfy $\exp(\bx)\exp(\by)\exp(\bz)=1_G$, then
$$\|\bx+\by+\bz\|\le\varepsilon\min\{\|\bx\|,\|\by\|,\|\bz\|,\|\bx+\by\|,\|\by+\bz\|,\|\bz+\bx\|\}.$$
\end{lemma}

\begin{proof}
By symmetry, it suffices to prove that
\begin{equation}\label{E:Lie0}
\|\bx+\by+\bz\|\le\varepsilon\min\{\|\bx\|,\|\by+\bz\|\}.
\end{equation}
Consider the map $\Phi:B_\Lg(\tau_2)\times B_\Lg(\tau_2)\to\Lg$ given by
\begin{equation}\label{E:Lie03}
\Phi(\bx,\by)=\log\big(\exp(\bx)\exp(\by)\big)-\bx-\by.
\end{equation}
Note that
\begin{equation}\label{E:Lie04}
\Phi(\bx,0)=\Phi(0,\by)=0, \qquad \forall \, \bx\in B_\Lg(\tau_2).
\end{equation}
Thus, if we let $\frac{\partial\Phi}{\partial\bx}:B_\Lg(\tau_2)\times B_\Lg(\tau_2)\to\sL(\Lg)$ be the partial derivative 
of $\Phi$ with respect to $\bx$, then
$$\Phi(\bx,\by)=\left(\int_0^1\frac{\partial\Phi}{\partial\bx}(t\bx,\by)dt\right)\bx.$$
Note that $\frac{\partial\Phi}{\partial\bx}$ is continuous, and it follows from \eqref{E:Lie04} that $\frac{\partial\Phi}{\partial\bx}(0,0)=0$. Thus, for any $\varepsilon>0$, there exists $\tau_3\in(0,\tau_2]$ such that
$$\bx,\by\in B_\Lg(\tau_3) \quad \Longrightarrow \quad \left\|\frac{\partial\Phi}{\partial\bx}(\bx,\by)\right\|\le\frac{\varepsilon}{1+\varepsilon}.$$
Therefore,
$$\|\Phi(\bx,\by)\|\le\left(\int_0^1\left\|\frac{\partial\Phi}{\partial\bx}(t\bx,\by)\right\|dt\right)\|\bx\|\le\frac{\varepsilon}{1+\varepsilon}\|\bx\|, \qquad \forall \, \bx,\by\in B_\Lg(\tau_3).$$
Suppose that $\bx,\by,\bz\in B_\Lg(\tau_3)$ and $\exp(\bx)\exp(\by)\exp(\bz)=1_G$. Then $\Phi(\bx,\by)=-(\bx+\by+\bz)$. It follows that
\begin{equation}\label{E:Lie01}
\|\bx+\by+\bz\|=\|\Phi(\bx,\by)\|\le\frac{\varepsilon}{1+\varepsilon}\|\bx\|.
\end{equation}
This in turn implies that
\begin{align}
\|\bx+\by+\bz\|&=(1+\varepsilon)\|\bx+\by+\bz\|-\varepsilon\|\bx+\by+\bz\| \notag \\
&\le\varepsilon\|\bx\|-\varepsilon\|\bx+\by+\bz\| \notag \\
&\le\varepsilon\|\by+\bz\|. \label{E:Lie02}
\end{align}
Now \eqref{E:Lie0} follows from \eqref{E:Lie01} and \eqref{E:Lie02}.
\end{proof}

For the convenience of later reference, let us record the following corollary.

\begin{corollary}\label{C:Lie1}
For any $\varepsilon>0$, there exists $\tau_4=\tau_4(\varepsilon)\in(0,\tau_2]$ such that
\begin{itemize}
  \item[(1)] For any $\bx,\by\in B_\Lg(\tau_4)$, we have
$$\|\log\big(\exp(\bx)\exp(\by)\big)\|\le(1+\varepsilon)\|\bx+\by\|.$$
  \item[(2)] For any $\bx,\by,\bz\in B_\Lg(\tau_4)$, we have
$$\|\log\big(\exp(\bx)\exp(\by)\exp(\bz)\big)-\by\|\le(1+\varepsilon)(\|\bx\|+\|\bz\|).$$
\end{itemize}
\end{corollary}

\begin{proof}
For $\varepsilon>0$, let $\tau_3=\tau_3(\varepsilon)\in(0,\tau_2]$ be as in Lemma \ref{L:Lie0}, and let $\tau_4\in(0,\tau_2]$ be such that
$$\bx_1,\bx_2,\bx_3\in B_\Lg(\tau_4) \quad \Longrightarrow \quad \log\big(\exp(\bx_1)\exp(\bx_2)\exp(\bx_3)\big)\in B_\Lg(\tau_3).$$
Then (1) follows by applying Lemma \ref{L:Lie0} to $\bz=-\log\big(\exp(\bx)\exp(\by)\big)$.
For (2), let
$$\bw=\log\big(\exp(\bx)\exp(\by)\exp(\bz)\big), \qquad \bv=\log\big(\exp(\bx)\exp(\by)\big).$$
Then $\bw,\bv\in B_\Lg(\tau_3)$. Note that
$$\exp(\bv)\exp(-\by)\exp(-\bx)=\exp(\bw)\exp(-\bz)\exp(-\bv)=1_G.$$
It follows from Lemma \ref{L:Lie0} that
\begin{align*}
\|\bw-\by\|&=\|(\bv-\by-\bx)+(\bw-\bz-\bv)+(\bx+\bz)\|\\
&\le\|\bv-\by-\bx\|+\|\bw-\bz-\bv\|+\|\bx+\bz\|\\
&\le\varepsilon\|\bx\|+\varepsilon\|\bz\|+\|\bx+\bz\|\\
&\le(1+\varepsilon)(\|\bx\|+\|\bz\|).
\end{align*}
This proves (2).
\end{proof}

We will only use the $\varepsilon=1$ case of Corollary \ref{C:Lie1}. However, the following result will be needed for arbitrarily small $\varepsilon$.

\begin{lemma}\label{L:Lie2}
Let $\Lh$ be a subalgebra of $\Lg$. Then there exist $\tau_5\in(0,\tau_2]$ and a function $\delta_1:(0,1)\to(0,\tau_2]$ such that for any $\by\in B_\Lh(\tau_5)$, there exists $T_\by\in\GL(\Lh)$ with $\|T_\by\|\le2$ such that
$$\varepsilon\in(0,1), \bx\in B_\Lh\big(\delta_1(\varepsilon)\big) \quad \Longrightarrow \quad \|\log\big(\exp(\bx)\exp(\by)\big)-\by-T_\by\bx\|\le\varepsilon\|\bx\|.$$
\end{lemma}

\begin{proof}
The map $\Phi$ defined in \eqref{E:Lie03} sends $B_\Lh(\tau_2)\times B_\Lh(\tau_2)$ into $\Lh$. Let $\Phi_\Lh$ be the restriction of $\Phi$ to $\Lh\times \Lh$. For $\by\in B_\Lh(\tau_2)$, let
$$T_\by=\id_\Lh+\frac{\partial\Phi_\Lh}{\partial\bx}(0,\by).$$
Then for $\bx\in B_\Lh(\tau_2)$, we have
\begin{align*}
\log\big(\exp(\bx)\exp(\by)\big)-\by-T_\by\bx&=\Phi_\Lh(\bx,\by)-\frac{\partial\Phi_\Lh}{\partial\bx}(0,\by)\bx\\
&=\left(\int_0^1\left(\frac{\partial\Phi_\Lh}{\partial\bx}(t\bx,\by)-\frac{\partial\Phi_\Lh}{\partial\bx}(0,\by)\right)dt\right)\bx.
\end{align*}
Since the map $\frac{\partial\Phi_\Lh}{\partial\bx}:B_\Lh(\tau_2)\times B_\Lh(\tau_2)\to\sL(\Lh)$ is continuous, it is uniformly continuous. Hence, there exists a function $\delta_1:(0,1)\to(0,\tau_2]$ such that for any $\varepsilon\in(0,1)$ and $\bx,\bx',\by,\by'\in  B_\Lh(\tau_2)$, we have
$$\max\{\|\bx-\bx'\|,\|\by-\by'\|\}\le\delta_1(\varepsilon) \quad \Longrightarrow \quad \left\|\frac{\partial\Phi_\Lh}{\partial\bx}(\bx,\by)-\frac{\partial\Phi_\Lh}{\partial\bx}(\bx',\by')\right\|\le\varepsilon.$$
Let $\tau_5=\delta_1(1/2)$. In view of $\frac{\partial\Phi_\Lh}{\partial\bx}(0,0)=0$, it follows that for any $\by\in B_\Lh(\tau_5)$, we have $\|\frac{\partial\Phi_\Lh}{\partial\bx}(0,\by)\|\le1/2$, and hence $T_\by$ is invertible and $\|T_\by\|\le2$. Moreover, it follows that if $\bx\in B_\Lh\big(\delta_1(\varepsilon)\big)$, then
$$\left\|\log\big(\exp(\bx)\exp(\by)\big)-\by-T_\by\bx\right\|\le\left(\int_0^1\left\|\frac{\partial\Phi_\Lh}{\partial\bx}(t\bx,\by)
-\frac{\partial\Phi_\Lh}{\partial\bx}(0,\by)\right\|dt\right)\|\bx\|
\le\varepsilon\|\bx\|.$$
This completes the proof.
\end{proof}

\subsection{A nice neighborhood of $Z$}\label{z}

Let now $\Gamma$ be a discrete subgroup of $G$, and let $X=G/\Gamma$. For $x\in X$, we define the map
$$\exp_x:\Lg\to X, \qquad \exp_x(\bx)=\exp(\bx)x.$$
Let $d\exp_x:\Lg\to T_xX$ be the tangent map of $\exp_x$ at $0$.
The next lemma shows the existence of a good neighborhood of the submanifold $Z$ of $X$ whenever $Z$ is compact.

\begin{lemma}\label{L:nbhd}
Let $Z\subset X$ be a compact $C^1$ submanifold (possibly with boundary). For $z\in Z$, consider the subspace of $\Lg$ given by $W_z=(d\exp_z)^{-1}(T_zZ)$. Then there exists a function $\delta_2:(0,1)\to(0,\infty)$ such that for any $\varepsilon\in(0,1)$ and $r\in(0,\delta_2(\varepsilon)]$, there exists a neighborhood $\Omega$ of $Z$ satisfying the following property: For any $y\in\Omega$, there exists $z\in Z$ such that
$$\bx\in B_\Lg(r),\  \exp_y(\bx)\in\Omega \quad \Longrightarrow \quad \dist(\bx,W_z)<\varepsilon r.$$
\end{lemma}

\begin{proof}
First, let us notice that there exists a function $\delta_3:(0,1)\to(0,\infty)$ such that for any $\varepsilon\in(0,1)$ and $z\in Z$, we have
\begin{equation}\label{E:tangent}
\by\in B_\Lg\big(\delta_3(\varepsilon)\big),\  \exp_z(\by)\in Z \quad \Longrightarrow \quad \dist(\by,W_z)\le\varepsilon\|\by\|.
\end{equation}
In fact, since $Z$ is compact, there exists $\tau_6>0$ such that for every $z\in Z$, there is a unique $C^1$ map $\phi_z:B_{W_z}(\tau_6)\to W_z^\bot$ with $\phi_z(0)=0$ satisfying the following property: If $\by\in B_\Lg(\tau_6)$ and $\exp_z(\by)\in Z$, then $\by=P_z\by+\phi_z(P_z\by)$, where $P_z$ is the orthogonal projection from $\Lg$ onto $W_z$.

Let $(d\phi_z)_\bw:W_z\to W_z^\bot$ be the tangent map of $\phi_z$ at $\bw\in B_{W_z}(\tau_6)$. Then $(d\phi_z)_0=0$, and the map $(z,\bw)\mt (d\phi_z)_\bw$ (as a map between bundles over $Z$ whose fibers at $z$ are $B_{W_z}(\tau_6)$ and the space of linear maps $W_z\to W_z^\bot$, respectively) is continuous. It follows that there exists a function $\delta_3:(0,1)\to(0,\tau_6]$ such that for any $z\in Z$, $\varepsilon\in(0,1)$ and $\bw\in B_{W_z}\big(\delta_3(\varepsilon)\big)$, we have $\|(d\phi_z)_\bw\|\le\varepsilon$, and hence $\|\phi_z(\bw)\|\le\varepsilon\|\bw\|$. Now, if $\by\in B_\Lg\big(\delta_3(\varepsilon)\big)$ and $\exp_z(\by)\in Z$, then
$$\dist(\by,W_z)=\|\by-P_z\by\|=\|\phi_z(P_z\by)\|\le\varepsilon\|P_z\by\|\le\varepsilon\|\by\|.$$
Hence \eqref{E:tangent} holds.
\smallskip

Define the function $\delta_2$ as
$$\delta_2(\varepsilon)=\min\big\{\delta_3(\varepsilon/4)/2,\tau_4(1)\big\},$$
where $\tau_4(\cdot)$ is as in Corollary \ref{C:Lie1}. Let $\varepsilon\in(0,1)$, $r\in(0,\delta_2(\varepsilon)]$. We verify that the neighborhood
$$\Omega=\bigcup_{z\in Z}\exp_z\big(B_\Lg^\circ(\varepsilon r/8)\big)$$
of $Z$ satisfies the required property. Let $y\in\Omega$. Then there exists $z\in Z$ such that $y=\exp_z(\bv)$ for some $\bv\in B_\Lg^\circ(\varepsilon r/8)$. Suppose $\bx\in B_\Lg(r)$ and $\exp_y(\bx)\in\Omega$. We need to show that $\dist(\bx,W_z)<\varepsilon r$.
Since $\exp_y(\bx)\in\Omega$, there exist $z'\in Z$ and $\bv'\in B_\Lg^\circ(\varepsilon r/8)$ such that $\exp_y(\bx)=\exp_{z'}(\bv')$, that is, $\exp(\bx)y=\exp(\bv')z'$.
Since
$$\|\bx\|\le r\le\delta_2(\varepsilon)\le\tau_4(1)$$ and
$$\max\big\{\|\bv\|,\|\bv'\|\big\}<\varepsilon r/8\le\tau_4(1),$$
if we write
$$\by=\log\big(\exp(-\bv')\exp(\bx)\exp(\bv)\big),$$
then it follows from Corollary \ref{C:Lie1}(2) that
\begin{equation}\label{E:exp31}
\|\by\|\le\|\bx\|+2\big(\|\bv\|+\|\bv'\|\big)<r+\frac{\varepsilon r}{2}<2r\le2\delta_2(\varepsilon)\le\delta_3(\varepsilon/4)
\end{equation}
and
\begin{equation}\label{E:exp32}
\|\bx-\by\|\le2\big(\|\bv\|+\|\bv'\|\big)<\frac{\varepsilon r}{2}.
\end{equation}
Note that
$$\exp_z(\by)=\exp(-\bv')\exp(\bx)\exp(\bv)z=\exp(-\bv')\exp(\bx)y=z'\in Z.$$
Thus, it follows from \eqref{E:tangent} and \eqref{E:exp31} that
$$\dist(\by,W_z)\le\frac{\varepsilon}{4}\|\by\|<\frac{\varepsilon r}{2}.$$
Hence, by \eqref{E:exp32}, we have
$$\dist(\bx,W_z)\le\|\bx-\by\|+\dist(\by,W_z)<\varepsilon r.$$
This proves the lemma.
\end{proof}

We now prove:

\begin{lemma}\label{L:main}
Let $G, \Gamma, X, H, f$ and $Z$ be as in Theorem \ref{T:main-discrete}, and assume the conditions in the theorem hold. Moreover, assume that $Z$ is compact (possibly with boundary). Then there exist $\tau_7\in(0,\tau_1]$ and a function $\tilde{r}_0:(0,1)\to(0,\infty)$ such that for any $\varepsilon\in(0,1)$ and $r_0\in(0,\tilde{r}_0(\varepsilon)]$, there exists a neighborhood $\Omega=\Omega(\varepsilon,r_0)$ of $Z$ satisfying the following property: For any $x\in X$, any closed ball $B\subset B_\Lh^\circ(\tau_7)$ of radius $r\le r_0$ and any $n\ge0$ with
\begin{equation}\label{E:Ad-ineq}
\frac{\varepsilon r_0}{r}\le\|(d\sigma_f)^n\|\le\frac{r_0}{r},
\end{equation}
there exists an affine hyperplane $L=L(x,B,n)$ in $\Lh$ such that
\begin{equation}\label{E:inc}
\exp_x^{-1}\big(f^{-n}(\Omega)\big)\cap B\subset L^{(\varepsilon r)},
\end{equation}
where $L^{(\varepsilon r)}$ is the $\varepsilon r$-neighborhood of $L$ in $\Lh$.
\end{lemma}

Note that in the statement of Lemma \ref{L:main}, we do not require that $\exp_x$ is injective on $B_\Lh^\circ(\tau_7)$.

\begin{proof}
For $z\in Z$, let $W_z$ be the subspace of $\Lg$ given in Lemma \ref{L:nbhd}. We want to apply Proposition \ref{P:linear-main} to $V=\Lg$, $U=\Lh$, $T=d\sigma_f$ and $\sW=\{W_z:z\in Z\}$. Since $Z$ is $C^1$, the map $Z\to\Gr_{\dim Z}(\Lg)$, $z\mt W_z$ is continuous. It then follows from the compactness of $Z$ that $\sW$ is compact. Condition (i) (resp.\ (ii)) in Theorem \ref{T:main-discrete} implies condition (i) (resp.\ (ii)) in Proposition \ref{P:linear-main}, and also condition \eqref{E:discrete-main3} implies \eqref{E:linear-main}. Thus, all conditions in Proposition \ref{P:linear-main} hold. It follows that there exist $c>0$ such that for any $z\in Z$ and $n\ge0$, there exists a linear hyperplane $L_{z,n}$ in $\Lh$ with
\begin{equation}\label{E:main-inq}
\dist\big((d\sigma_f)^n\bx,W_z\big)\ge c\|(d\sigma_f)^n\|\dist(\bx,L_{z,n}), \qquad \forall \,  \bx\in\Lh.
\end{equation}
Let
$$\tau_7=\min\{\tau_4(1),\tau_5\},$$
where $\tau_4(\cdot)$ and $\tau_5$ are as in Corollary \ref{C:Lie1} and Lemma \ref{L:Lie2}, and define the function $\tilde{r}_0$ as
$$\tilde{r}_0(\varepsilon)=\frac14\min\left\{\delta_1(\varepsilon/8),\delta_2(c\varepsilon^2/4)\right\},$$
where $\delta_1(\cdot)$ and $\delta_2(\cdot)$ are as in Lemmas \ref{L:Lie2} and \ref{L:nbhd}. Let $\varepsilon\in(0,1)$, $r_0\in(0,\tilde{r}_0(\varepsilon)]$. By Lemma \ref{L:nbhd} and the choice of $\tilde{r}_0(\varepsilon)$, there exists a neighborhood $\Omega$ of $Z$ such that for any $y_0\in\Omega$, there is $z_0\in Z$ such that
\begin{equation}\label{E:dist-m}
\bv\in B_\Lg(4r_0), \exp_{y_0}(\bv)\in\Omega \quad \Longrightarrow \quad \dist(\bv,W_{z_0})<\frac{c\varepsilon^2}{4}r_0.
\end{equation}
In what follows, we prove that $\Omega$ satisfies the required property in Lemma \ref{L:main}.

Let $x\in X$, $B\subset B_\Lh^\circ(\tau_7)$ be a closed ball of radius $r\le r_0$, and $n\ge0$ satisfy \eqref{E:Ad-ineq}. We need to show that there exists an affine hyperplane $L\subset \Lh$ satisfying \eqref{E:inc}. Without loss of generality, assume that $\exp_x^{-1}\big(f^{-n}(\Omega)\big)\cap B\ne\varnothing$. We choose and fix a point $\by_0\in\exp_x^{-1}\big(f^{-n}(\Omega)\big)\cap B$. Let $y_0=f^n\big(\exp_x(\by_0)\big)\in\Omega$, and let $z_0\in Z$ satisfy \eqref{E:dist-m}. Since $\by_0\in B\subset B_\Lh^\circ(\tau_7)$ and $\tau_7\le\tau_5$, it follows from Lemma \ref{L:Lie2} and the choice of $\tilde{r}_0(\varepsilon)$ that there exists $T_{\by_0}\in \GL(\Lh)$ with $\|T_{\by_0}\|\le2$ such that
\begin{equation}\label{E:tangent-m}
\bx\in B_\Lh\big(4\tilde{r}_0(\varepsilon)\big) \quad \Longrightarrow \quad \left\|\log\big(\exp(\bx)\exp(\by_0)\big)-\by_0-T_{\by_0}\bx\right\|\le\frac{\varepsilon}{8}\|\bx\|.
\end{equation}
We verify that the hyperplane
$$L=\by_0+T_{\by_0}(L_{z_0,n})$$
satisfies \eqref{E:inc}.

Let $\by\in \exp_x^{-1}\big(f^{-n}(\Omega)\big)\cap B$. We need to prove that $\by\in L^{(\varepsilon r)}$. Let $y=f^n\big(\exp_x(\by)\big)\in\Omega$, $\bx=\log\big(\exp(\by)\exp(-\by_0)\big)\in\Lh$. Since $\by_0,\by\in B$, we have $\|\by-\by_0\|\le 2r$. Note also that $B\subset B_\Lh^\circ(\tau_7)$ and $\tau_7\le\tau_4(1)$. It then follows from Corollary \ref{C:Lie1}(1) that
$$\|\bx\|\le2\|\by-\by_0\|\le4r.$$
Thus
$$\left\|(d\sigma_f)^n\bx\right\|\le\left\|(d\sigma_f)^n\right\|\|\bx\|\le\frac{r_0}{r}\cdot4r=4r_0.$$
Note also that
\begin{align*}
\exp_{y_0}\big((d\sigma_f)^n\bx\big)&=\exp\big((d\sigma_f)^n\bx\big)y_0=\sigma_f^n\big(\exp(\bx)\big)f^n(\exp(\by_0)x)\\
&\stackrel{\eqref{E:affine2}}=f^n(\exp(\bx)\exp(\by_0)x)=f^n(\exp(\by)x)=y\in\Omega.
\end{align*}
Hence, it follows from the choice of $z_0$ that
$$\dist\big((d\sigma_f)^n\bx,W_{z_0}\big)<\frac{c\varepsilon^2}{4}r_0.$$
Together with \eqref{E:main-inq}, this implies that
\begin{align*}
\dist(\bx,L_{z_0,n})&\le c^{-1}\left\|(d\sigma_f)^n\right\|^{-1}\dist\big((d\sigma_f)^n\bx,W_{z_0}\big)\\
&<c^{-1}\cdot\frac{r}{\varepsilon r_0}\cdot\frac{c\varepsilon^2}{4}r_0=\frac{\varepsilon r}{4}.
\end{align*}
Hence, if we let $\bz\in L_{z_0,n}$ be such that $\|\bx-\bz\|=\dist(\bx,L_{z_0,n})$, then
\begin{align*}
\dist(T_{\by_0}\bx,T_{\by_0}L_{z_0,n})&\le\|T_{\by_0}\bx-T_{\by_0}\bz\|\le\|T_{\by_0}\|\|\bx-\bz\|\\
&\le2\dist(\bx,L_{z_0,n})<\frac{\varepsilon r}{2}.
\end{align*}
On the other hand, since $\|\bx\|\le4r\le4r_0\le4\tilde{r}_0(\varepsilon)$ and $\exp(\by)=\exp(\bx)\exp(\by_0)$, it follows from \eqref{E:tangent-m} that
$$\|\by-\by_0-T_{\by_0}\bx\|\le\frac{\varepsilon}{8}\|\bx\|\le\frac{\varepsilon r}{2}.$$
This implies that
\begin{align*}
\dist(\by,L)&=\dist\big(\by-\by_0,T_{\by_0}(L_{z_0,n})\big)\\
&\le\dist\big(T_{\by_0}\bx,T_{\by_0}(L_{z_0,n})\big)+\|\by-\by_0-T_{\by_0}\bx\|\\
&<\frac{\varepsilon r}{2}+\frac{\varepsilon r}{2}\le\varepsilon r.
\end{align*}
Hence $\by\in L^{(\varepsilon r)}$. This completes the proof.
\end{proof}

\subsection{Proof of Theorem \ref{T:main-discrete}}
We now use Lemma \ref{L:main} to prove Theorem \ref{T:main-discrete}.

\begin{proof}[Proof of Theorem \ref{T:main-discrete}]
Since any $C^1$ submanifold of $X$ is the union of countably many compact $C^1$ submanifolds (possibly with boundaries), we may assume without loss of generality that $Z$ is compact. Let $x\in X$. We need to prove that for every $h_0\in H$, there is an open neighborhood $U$ of $h_0$ in $H$ such that the set $\{h\in U:\overline{\{f^n(hx):n\ge0\}}\cap Z=\emptyset\}$
is HAW on $U$. By replacing $x$ with $h_0x$, we may assume that $h_0=1_G$. Let $\tau_7>0$ be as in Lemma \ref{L:main}, and let $U=\exp\big(B_\Lh^\circ(\tau_7)\big)$. Since the exponential map restricts to a diffeomorphism from $B_\Lh^\circ(\tau_7)$ onto $U$, in view of Lemma \ref{L:HPW2}, it suffices to prove that the set
\begin{equation}\label{E:target}
\left\{\bx\in B_\Lh^\circ(\tau_7):\overline{\{f^n\big(\exp_x(\bx)\big):n\ge0\}}\cap Z=\emptyset\right\}
\end{equation}
is HPW on $B_\Lh^\circ(\tau_7)$.
\smallskip

Let $\beta\in\big(0,\beta_0(\dim\Lh)\big)$ be fixed. By Lemma \ref{L:linear1}, there exists $C>1$ such that
\begin{equation}\label{E:mo}
C^{-1}n^{s-1}\rho^n\le\|(d\sigma_f)^n\|\le Cn^{s-1}\rho^n, \qquad \forall \,  n\ge0,
\end{equation}
where $\rho=\rho(d\sigma_f)>1$ and $s=s(d\sigma_f)$. Let $\ell\in\NN$ be large such that
\begin{equation}\label{E:n1}
\rho^{2^\ell-1}\beta^\ell\ge C^2
\end{equation}
and
\begin{equation}\label{E:n2}
C\beta^\ell\le1.
\end{equation}
We use Lemma \ref{L:main} with $\varepsilon=\beta^{\ell+1}$ to describe a winning strategy for Alice when playing the $\beta$-hyperplane percentage game on $B_\Lh^\circ(\tau_7)$ with target set \eqref{E:target}. As remarked in \S \ref{S:HPW}, we may assume that Bob will play so that $r_i\to0$ and $r_0\le\tilde{r}_0(\beta^{\ell+1})$, where $\tilde{r}_0(\cdot)$ is as in Lemma \ref{L:main}. Let us partition the game into stages. For $k\ge0$, we define the $k$th stage to be the set of indices $i\ge0$ for which
\begin{equation}\label{E:i0}
\beta^{\ell(k+1)}r_0<r_i\le\beta^{\ell k}r_0.
\end{equation}
Then each stage is finite and contains at least $\ell$ indices.
Suppose that the $k$th stage starts when Bob chooses the ball $B_{i_k}$ in $\Lh$, that is, $i_k$ is the smallest index in the $k$th stage. In particular, we have $i_0=0$.
It follows from the rule of the game that
\begin{equation}\label{E:i}
\beta^{\ell k+1}r_0<r_{i_k}\le\beta^{\ell k}r_0.
\end{equation}
Consider the set of integers
\begin{equation}\label{E:Nk}
\cN_k=\left\{n\ge0:\beta^{-\ell(k-1)}\le \|(d\sigma_f)^n\|<\beta^{-\ell k}\right\}.
\end{equation}
Note that by \eqref{E:mo} and \eqref{E:n2}, we have $\|(d\sigma_f)^n\|\ge C^{-1}\ge\beta^\ell$ for any $n\ge0$. Thus
\begin{equation}\label{E:union}
\bigcup_{k\ge0}\cN_k=\NN\cup\{0\}.
\end{equation}
Note also that for any $n_1,n_2\in\cN_k$ with $n_1<n_2$, we have
$$C^{-2}\rho^{n_2-n_1}\le\frac{C^{-1}n_2^{s-1}\rho^{n_2}}{Cn_1^{s-1}\rho^{n_1}}\stackrel{\eqref{E:mo}}
\le\frac{\|(d\sigma_f)^{n_2}\|}{\|(d\sigma_f)^{n_1}\|}\stackrel{\eqref{E:Nk}}<\frac{\beta^{-\ell k}}{\beta^{-\ell(k-1)}}
=\beta^{-\ell} \stackrel{\eqref{E:n1}}\le C^{-2}\rho^{2^\ell-1},$$
which implies that $n_2-n_1<2^\ell-1$. Hence
\begin{equation}\label{E:stage2}
\#\cN_k<2^\ell.
\end{equation}
It follows from \eqref{E:i} and \eqref{E:Nk} that if $n\in\cN_k$, then
$$r_{i_k}\|(d\sigma_f)^n\|\in[\beta^{\ell k+1}r_0\cdot\beta^{-\ell(k-1)},\beta^{\ell k}r_0\cdot\beta^{-\ell k}]=[\beta^{\ell+1}r_0,r_0].$$
Therefore, if we let $\Omega=\Omega(\beta^{\ell+1},r_0)$ be the neighborhood of $Z$ given by Lemma \ref{L:main}, then for any $n\in\cN_k$, there exists an affine hyperplane $L(B_{i_k},n)$ in $\Lh$ such that
\begin{equation}\label{E:nbhd0}
\exp_x^{-1}\big(f^{-n}(\Omega)\big)\cap B_{i_k}\subset L(B_{i_k},n)^{(\beta^{\ell+1}r_{i_k})}.
\end{equation}
Let Alice's $i_k$-th move be the hyperplane neighborhoods
\begin{equation}\label{E:nbhd}
\left\{L(B_{i_k},n)^{(\beta^{\ell+1}r_{i_k})}:n\in\cN_k\right\}.
\end{equation}
More generally, for any index $i$ in the $k$th stage, after Bob choosing the ball $B_i$, let Alice choose those neighborhoods in \eqref{E:nbhd} which intersect $B_i$. Note that
$$\beta^{\ell+1}r_{i_k} \stackrel{\eqref{E:i}}\le \beta^{\ell+1}\cdot\beta^{\ell k}r_0=\beta\cdot\beta^{\ell(k+1)}r_0 \stackrel{\eqref{E:i0}}<\beta r_i.$$
So Alice's moves are legal. We prove that this strategy guarantees a win for Alice.
\smallskip

In view of the rule of the game, it follows that if $i$ is an index in the $k$th stage, then
\begin{equation}\label{E:stage}
\#\left\{n\in\cN_k:B_{i+1}\cap L(B_{i_k},n)^{(\beta^{\ell+1}r_{i_k})}\ne\emptyset\right\}\le\frac{\#\cN_k}{2^{i+1-i_k}}\stackrel{\eqref{E:stage2}}<2^{\ell-(i+1-i_k)}.
\end{equation}
On the other hand, since each stage contains at least $\ell$ indices, the index $i_k+\ell-1$ is in the $k$th stage. Substituting $i=i_k+\ell-1$ into \eqref{E:stage}, we obtain
$$\#\left\{n\in\cN_k:B_{i_k+\ell}\cap L(B_{i_k},n)^{(\beta^{\ell+1}r_{i_k})}\ne\emptyset\right\}<1.$$
This means that
$$B_{i_k+\ell}\cap L(B_{i_k},n)^{(\beta^{\ell+1}r_{i_k})}=\emptyset  \qquad \forall \, n\in\cN_k.$$
Together with \eqref{E:nbhd0}, this implies that
$$B_{i_k+\ell}\cap\exp_x^{-1}\big(f^{-n}(\Omega)\big)=\emptyset  \qquad \forall \, n\in\cN_k.$$
Hence, for any $n\in\cN_k$, the unique point $\bx_\infty$ in $\bigcap_{i=0}^\infty B_i$ is not contained in $\exp_x^{-1}\big(f^{-n}(\Omega)\big)$, or equivalently, $f^n\big(\exp_x(\bx_\infty)\big)\notin\Omega$. In view of \eqref{E:union}, it follows that $\bx_\infty$ is contained in the target set \eqref{E:target}. Hence Alice wins.
\end{proof}

\begin{remark}\label{R:omega2}
In view of Remark \ref{R:omega1}, it follows that without condition \eqref{E:discrete-main3}, Lemma \ref{L:main} remains valid for sufficient large $n$. This implies that if condition \eqref{E:discrete-main3} is dropped and the submanifold $Z$ in Theorem \ref{T:main-discrete} is compact, then there exists $N=N(Z)$ such that the set $\left\{h\in H:\overline{\{f^n(hx):n\ge N\}}\cap Z=\emptyset\right\}$ is HAW on $H$, and hence for a general submanifold $Z$, the set $\{h\in H:\omega(hx)\cap Z=\emptyset\}$ is HAW on $H$. In turn, if condition \eqref{E:continuous-main4} in Theorem \ref{T:main-continuous} is dropped, the set $\{h\in H:\omega(hx)\cap Z=\emptyset\}$ is HAW on $H$.
\end{remark}

\section{Geodesic flows on locally symmetric spaces}\label{S:geo}

This section is devoted to the proof of Theorems \ref{B1} and \ref{B2}. We first use Theorem \ref{T:main-continuous} (and its proof) to prove a result on semisimple Lie groups.

\subsection{A proposition on semisimple groups}\label{S:ss}

Let $G$ be a noncompact semisimple Lie group with finitely many connected components, $K\subset G$ a maximal compact subgroup, $\Lg$ and $\Lk$ the Lie algebras of $G$ and $K$  respectively, and $\Lp$ the orthogonal complement of $\Lk$ in $\Lg$ with respect to the Killing form on $\Lg$. We assume that the identity component $G^\circ$ of $G$ has finite center. Then $\Lg=\Lk\oplus\Lp$ is a Cartan decomposition. Note that the identity component $K^\circ$ of $K$ is a maximal compact subgroup of $G^\circ$.

\begin{proposition}\label{P:ss}
Let $G$, $K$ and $\Lp$ be as above, $\Gamma\subset G$ a discrete subgroup, $X=G/\Gamma$, $\bv\in\Lp\smz$, and $F=\{g_t:t\in\RR\}$ the one-parameter subgroup given by $g_t=\exp(t\bv)$.
\begin{itemize}
  \item[(1)] Let $x_1,x_2\in X$ be such that $Kx_1\ne Kx_2$. Then the set $$\{k\in K:kx_1\in E(F,Kx_2)\}$$ is HAW on $K$.
  \item[(2)] Let $K'\subset K$ be a closed subgroup with $\dim K'<\dim K$, and $S\subset X$ be a finite subset. Then there exists an $F$-invariant closed subset of $X$ that does not intersect $K'S$ but intersects every $K^\circ$-orbit in $X$.
\end{itemize}
\end{proposition}

\begin{proof}
(1) For a subset $A\subset\RR$, let us denote $F_A=\{g_t:t\in A\}$. Since $K$ is compact and $Kx_1\ne Kx_2$, there exists $\varepsilon>0$ such that $F_{[-\varepsilon,\varepsilon]}Kx_1\cap Kx_2=\varnothing$. Then for $k\in K$, we have
$kx_1\in E(F,Kx_2)$ if and only if both $\overline{F_{[\varepsilon,\infty)}kx_1}\cap Kx_2$ and $\overline{F_{(-\infty,-\varepsilon]}kx_1}\cap Kx_2$ are empty. Hence, to prove part (1), it is enough to prove that the sets
\begin{equation}\label{E:+e}
\left\{k\in K:\overline{F_{[\varepsilon,\infty)}kx_1}\cap K x_2=\varnothing\right\}
\end{equation}
and
\begin{equation}\label{E:-e}
\left\{k\in K:\overline{F_{(-\infty,-\varepsilon]}kx_1}\cap K x_2=\varnothing\right\}
\end{equation}
are HAW on $K$.

Let us prove that the set \eqref{E:+e} is HAW. Note that for $k\in K$, we have
$$\overline{F_{[\varepsilon,\infty)}kx_1}\cap K x_2=g_\varepsilon\big(\overline{F^+kx_1}\cap g_{-\varepsilon}Kx_2\big).$$
Thus, it suffices to show that the set
\begin{equation}\label{E:+e2}
\left\{k\in K:kx_1\in E(F^+,g_{-\varepsilon}Kx_2)\right\}
\end{equation}
is HAW. By Theorem \ref{T:main-continuous}, we only need to verify that $\rho(\Ad g_1)>1$, $Z=g_{-\varepsilon}Kx_2$ is $F$-transversal, and that conditions \eqref{E:continuous-main3} and \eqref{E:continuous-main4} hold for $H=K$.
The latter three conditions translate respectively as
\begin{align}
\bv & \notin(\Ad g_{-\varepsilon}) \Lk, \label{E:sym-1}\\
\Lk_{F^+}^{\max} & \not\subset(\Ad g_{-\varepsilon})\Lk\oplus\RR\bv, \label{E:sym-2}\\
(\Ad g_t) \Lk & \not\subset(\Ad g_{-\varepsilon}) \Lk\oplus\RR\bv \qquad \forall \, t\ge0. \label{E:sym-3}
\end{align}

To verify these conditions, let $\La$ be a maximal abelian subspace of $\Lp$ containing $\bv$, and let $\Sigma\subset\La^*$ be the restricted root system of $(\Lg,\La)$. Then the set of eigenvalues of $\Ad g_1$ is $\left\{e^{\lambda(\bv)}:\lambda\in\Sigma\right\}\cup\{1\}$. Since $\bv\ne0$, we have $\omega:=\max_{\lambda\in\Sigma}\lambda(\bv)>0$. It follows that $\rho(\Ad g_1)=e^\omega>1$.
\smallskip

Next, notice that $\bv=(\Ad g_{-\varepsilon})\bv\in(\Ad g_{-\varepsilon})\Lp$. Hence \eqref{E:sym-1} is clear.
\smallskip

To verify \eqref{E:sym-2}, recall that $\Lk_{F^+}^{\max}=p(\Ad g_1)(\Lk)$, where $p$ is the polynomial given in \S \ref{sub:continuous}. Let $\Lg=\Lg_0\oplus\bigoplus_{\lambda\in\Sigma}\Lg_\lambda$ be the restricted root space decomposition. Then $p(\Ad g_1)$ is the projection onto $\bigoplus_{\lambda(\bv)=\omega}\Lg_\lambda$ along $\Lg_0\oplus\bigoplus_{\lambda(\bv)<\omega}\Lg_\lambda$. Let $\lambda_0\in\Sigma$ be such that $\lambda_0(\bv)=\omega$. We first claim that $\Lg_{\lambda_0}\subset\Lk_{F^+}^{\max}$. In fact, if $\theta$ is the Cartan involution of $\Lg$ corresponding to the Cartan decomposition $\Lg=\Lk\oplus\Lp$, then for any $\bw\in\Lg_{\lambda_0}$ we have $\theta\bw\in\Lg_{-\lambda_0}$ and $\bw+\theta\bw\in\Lk$, and hence $\bw=p(\Ad g_1)(\bw+\theta\bw)\in\Lk_{F^+}^{\max}$, proving the claim. On the other hand, it follows from the Iwasawa decomposition (relative to a set of positive roots containing $\lambda_0$) that $\Lg_{\lambda_0}\not\subset\Lk\oplus\RR\bv$. Applying $\Ad g_{-\varepsilon}$ to both sides, we obtain $\Lg_{\lambda_0}\not\subset(\Ad g_{-\varepsilon})\Lk\oplus\RR\bv$. This, together with $\Lg_{\lambda_0}\subset\Lk_{F^+}^{\max}$, implies \eqref{E:sym-2}.

\smallskip

We now verify \eqref{E:sym-3}. Suppose the contrary. Then there exists $t\ge0$ such that $(\Ad g_{t+\varepsilon})\Lk\subset\Lk\oplus\RR\bv$.
Since $K^\circ$ is a maximal compact subgroup of $G^\circ$, it is self-normalizing in $G^\circ$. It follows that $(\Ad g_{t+\varepsilon})\Lk\ne\Lk$.
Let $\bx\in\Lk$ be such that $(\Ad g_{t+\varepsilon})\bx\notin\Lk$. Then there exist $\by\in\Lk$ and $b\in\RR\smz$ such that $$(\Ad g_{t+\varepsilon})\bx=\by+b\bv.$$
Taking the Cartan involution $\theta$ on both sides, we obtain
$$(\Ad g_{t+\varepsilon}^{-1})\bx=\by-b\bv.$$
It follows that
$$(\Ad g_{t+\varepsilon})\bx-(\Ad g_{t+\varepsilon}^{-1})\bx=2b\bv.$$
Let $\kappa(\cdot,\cdot)$ be the Killing form on $\Lg$. Since $\kappa|_{\Lp\times\Lp}$ is positive definite, we have
\begin{align*}
0&\ne\kappa(2b\bv,\bv)=\kappa\big((\Ad g_{t+\varepsilon})\bx,\bv\big)-\kappa\big((\Ad g_{t+\varepsilon}^{-1})\bx,\bv\big)\\
&=\kappa\big(\bx,(\Ad g_{t+\varepsilon}^{-1})\bv\big)-\kappa\big(\bx,(\Ad g_{t+\varepsilon})\bv\big)=\kappa(\bx,\bv)-\kappa(\bx,\bv)=0,
\end{align*}
a contradiction. This completes the verification of the required conditions, thus proves the set \eqref{E:+e} is HAW. A similar argument with $\bv$ replaced by $-\bv$ shows that the set \eqref{E:-e} is also HAW. This completes the proof of part (1).
\smallskip

(2) The proof is similar to that of Theorem \ref{T:main-continuous}. Let us sketch the argument and leave the details to the reader. First, we pick $\tau>0$ such that $Z:=F_{[0,\tau]}K'S$ is a smooth submanifold of $X$, and such that \eqref{E:discrete-main2} and \eqref{E:discrete-main3} hold for both $(H,f)=(K^\circ,g_\tau)$ and $(H,f)=(K^\circ,g_\tau^{-1})$. Then the conditions of Theorem \ref{T:main-discrete} are satisfied for both cases. As in the proof of Theorem \ref{T:main-discrete}, it can be shown that there exist positive constants $\tau_7$, $\beta$, $r_0$ and a neighborhood $\Omega$ of $Z$ such that for every $x\in X$, Alice has a winning strategy for the $\beta$-hyperplane percentage game on $B_\Lk^\circ(\tau_7)$ with target set
\begin{equation}\label{E:newtarget}
B_\Lk^\circ(\tau_7)\sm\bigcup_{n\in\ZZ}\exp_x^{-1}(g_{n\tau}\Omega),
\end{equation}
provided Bob's initial ball $B_0$ has the prescribed radius $r_0$. (A major difference is that we are now working with both $f=g_\tau$ and $f=g_\tau^{-1}$ simultaneously. So we need to replace ``$n\ge0$" by ``$n\in\ZZ$" in the definition of $\cN_k$ in \eqref{E:Nk}, and to replace \eqref{E:n1} by the slightly stronger condition $\rho^{2^{\ell-1}-1}\beta^\ell\ge C^2$ so that \eqref{E:stage2} still holds.) In particular, the set \eqref{E:newtarget} is nonempty. This implies that the set $\bigcup_{n\in\ZZ}g_{n\tau}\Omega$ does not contain any $K^\circ$-orbit in $X$. On the other hand, it is straightforward to show that $\bigcap_{t\in[0,\tau]}g_t^{-1}\Omega$ contains an open neighborhood $U$ of $K'S$. This implies $FU\subset \bigcup_{n\in\ZZ}g_{n\tau}\Omega$. Then the $F$-invariant closed set $X\sm FU$ satisfies the requirement.
\end{proof}

\subsection{Proofs of Theorems \ref{B1} and \ref{B2}}

We first review some basic facts concerning locally symmetric spaces. Let $Y$ be a locally symmetric space of noncompact type, and let $\tilde{Y}$ be its universal cover. The isometry group $G$ of $\tilde{Y}$ has finitely many connected components, and its identity component is a semisimple Lie group without compact factors and with trivial center. Let $y_0\in Y$, and $\tilde{y}_0\in\tilde{Y}$ be a preimage of $y_0$. The stabilizer $K:=\Stab_G(\tilde{y}_0)$ is a maximal compact subgroup of $G$. We identify the globally symmetric space $\tilde{Y}$ with $K\backslash G$, and view the fundamental group $\Gamma:=\pi_1(Y)$ as a subgroup of $G$ via deck transformations. Then $Y$ can be identified with $K\backslash\ggm$.

Let $\Lg$, $\Lk$ and $\Lp$ be as in \S \ref{S:ss}. Then we have a natural identification $T_{y_0}Y\cong \Lp$. Let $\Lp_1$ be the unit sphere in $\Lp$ (with respect to the metric on $T_{y_0}Y$) centered at $0$, which is identified with $S_{y_0}(Y)$. For $\bv\in\Lp_1$, let $\gamma(\bv)$ denote the geodesic line in $Y$ through $y_0$ in the direction $\bv$. Then
$$\gamma(\bv)=\{K\exp(t\bv)\Gamma:t\in\RR\}.$$
Let us now prove Theorem \ref{B1}.

\begin{proof}[Proof of Theorem \ref{B1}]
Without loss of generality, we assume $y=y_0$. We need to prove that the set
\begin{equation}\label{E:geo-1}
\{\bv\in\Lp_1:\overline{\gamma(\bv)}\cap Z=\varnothing\}
\end{equation}
is thick in $\Lp_1$. Note that $\Ad(K)\Lp_1=\Lp_1$.  We first prove that for every $\bv\in\Lp_1$, the set
\begin{equation}\label{E:geo-2}
\left\{\bw\in\Ad(K^\circ)\bv:\overline{\gamma(\bw)}\cap Z=\varnothing\right\}
\end{equation}
is HAW on $\Ad(K^\circ)\bv$. To do this, let $x_0$ denote the point $\Gamma$ in $X:=G/\Gamma$, and consider the surjective map
$$q:X\to Y, \qquad q(gx_0)=Kg\Gamma.$$
Let $F=\{g_t:t\in\RR\}$, where $g_t=\exp(t\bv)$. We claim that the set \eqref{E:geo-2} is the image of the set
\begin{equation}\label{E:geo-3}
\big\{k\in K^\circ:kx_0\in E\big(F,q^{-1}(Z)\big)\big\}
\end{equation}
under the submersion $K^\circ\to\Ad(K^\circ)\bv$, $k\mt(\Ad k^{-1})\bv$. In fact, for $k\in K^\circ$ and $t\in\RR$, we have $$K\exp\big(t(\Ad k^{-1})\bv\big)\Gamma=Kg_tk\Gamma=q(g_tkx_0).$$
So $\gamma\big((\Ad k^{-1})\bv\big)=q(Fkx_0)$.
Since the map $q$ has compact fibers, it is a closed map. It follows that
$\overline{\gamma\big((\Ad k^{-1})\bv\big)}=q\big(\overline{Fkx_0}\big)$.
Thus, $\overline{\gamma\big((\Ad k^{-1})\bv\big)}\cap Z=\varnothing$ if and only if $\overline{Fkx_0}\cap q^{-1}(Z)=\varnothing$, that is, $kx_0\in E\big(F,q^{-1}(Z)\big)$. This verifies the claim. Since $q^{-1}(Z)$ is a countable union $K$-orbits in $X$ distinct from $Kx_0$, it follows from Proposition \ref{P:ss}(1) that the set \eqref{E:geo-3} is HAW on $K^\circ$. Then, by Lemma \ref{L:haw-proj}, the set \eqref{E:geo-2} is HAW on $\Ad(K^\circ)\bv$.
\vskip 5pt

To complete the proof, let us choose a maximal abelian subspace $\La\subset\Lp$ and an (open) Weyl chamber $\La^+\subset\La$. Let $M=Z_{K^\circ}(\La)$, $\La^+_1=\La^+\cap\Lp_1$. Then the map
$$\Phi:K^\circ/M\times\La^+_1\to\Lp_1, \qquad \Phi(kM,\bv)=(\Ad k)\bv$$
is a diffeomorphism onto on open dense subset of $\Lp_1$. The HAW property of the set \eqref{E:geo-2} implies that for each $\bv\in\La^+_1$, the intersection of the set \eqref{E:geo-1} with $\Phi(K^\circ/M\times\{\bv\})$ is thick in $\Phi(K^\circ/M\times\{\bv\})$. By Marstrand slicing theorem (see, for example, \cite[Lemma 1.4]{KM}), the intersection of \eqref{E:geo-1} with $\im \Phi$ is thick in $\im \Phi$, hence is also thick in $\Lp_1$. This proves Theorem \ref{B1}.
\end{proof}

Before proving Theorem \ref{B2}, let us recall more facts concerning the geodesic flow on the unit tangent bundle $S(Y)$ (see, for example, \cite{Mau,KM2}). We keep the notation as in the beginning of this subsection, and consider the natural $G$-action on $S(\tilde{Y})$.
We refer to a connected component of the image of a $G$-orbit in $S(\tilde{Y})$ under the covering map $S(\tilde{Y})\to S(Y)$ as an \textsl{ergodic submanifold} of $S(Y)$. Each ergodic submanifold is a closed submanifold of $S(Y)$ and is invariant under the geodesic flow. Note that every $G^\circ$-orbit in $S(\tilde{Y})$ meets $\Lp_1\cong S_{\tilde{y}_0}(\tilde{Y})$. The stabilizer of a vector $\bv\in\Lp_1$ in $G$ is equal to its centralizer $K_\bv$ in $K$. So the $G$-orbit of $\bv$ in $S(\tilde{Y})$ can be identified with $K_\bv\backslash G$, and its projection in $S(Y)$ can be identified with $K_\bv\backslash\ggm$. Under the latter identification, the restriction of the geodesic flow on $K_\bv\backslash\ggm$ is given by
$$\gamma_t(K_\bv g\Gamma)=K_\bv\exp(t\bv)g\Gamma, \qquad g\in G.$$
Let $\E_\bv\subset S(Y)$ denote the corresponding ergodic submanifold, namely,
$$\E_\bv:=\{K_\bv g\Gamma:g\in G^\circ\}.$$
Then every ergodic submanifold is of the form $\E_\bv$ for some $\bv\in\Lp_1$.

\begin{proof}[Proof of Theorem \ref{B2}]
Assume $\E=\E_\bv$, where $\bv\in\Lp_1$. We keep the notation as in the proof of Theorem \ref{B1}. Then the surjective map
$$q':X\to K_\bv\backslash\ggm, \qquad q'(gx_0)=K_\bv g\Gamma$$
intertwines the flow $(X,F)$ and the geodesic flow on ${K_\bv\backslash\ggm}$. Each fiber of $q'$ is a $K_\bv$-orbit in $X$.
Note that $\dim K_\bv<\dim K$. By Proposition \ref{P:ss}(2), there is an $F$-invariant closed subset $X'\subset X$ that does not intersect $q'^{-1}(\Xi)$ but intersects every $K^\circ$-orbit in $X$. It follows that the closed subset $q'(X')\cap\E$ of $\E$ is invariant under the geodesic flow and does not intersect $\Xi$. Moreover, the projection of $q'(X')\cap\E$ to $Y$ contains $q(X'\cap G^\circ x_0)$, which is the whole space $Y$. This completes the proof of Theorem \ref{B2}.
\end{proof}

\begin{remark}
Similar to (in fact, simpler than) the proofs of Proposition \ref{P:ss}(1) and Theorem \ref{B1}, it can be shown that for every ergodic submanifold $\E\subset S(Y)$, if $Z\subset\E$ is a countable subset, then the set $\{\xi\in \E: \overline{\gamma(\xi)}\cap Z=\emptyset\}$ is HAW on $\E$.
\end{remark}

\section{Gaps between values of functions at integer points}\label{forms}

\subsection{The general set-up} \label{setup}

Let $n\ge2$ be an integer, and let $C(\RR^n)$ denote the space of real-valued continuous functions on $\RR^n$.
For $\phi\in C(\RR^n)$ we will be studying the values of $\phi$ at  nonzero integer points $\ZZ^n_{\ne0}:=\ZZ^n\smz$.
It is an important question in number theory to know that for $\phi$ as above, whether
$\phi(\ZZ^n_{\ne0})$ is dense in its image $\phi(\RR^n)$, or perhaps it has a gap at a real number $a\in\RR$. Here, we say that $\phi(\ZZ^n_{\ne0})$ {\sl has a gap at $a$} if $\phi(\ZZ^n_{\ne0})\cap(a-\varepsilon,a+\varepsilon)=\emptyset$ for some $\varepsilon>0$. Clearly when $a \ne \phi(0)$ it is equivalent to $\phi(\ZZ^n)\cap(a-\varepsilon,a+\varepsilon)=\emptyset$ for some $\varepsilon>0$.

If $\phi$ is a linear form, it is easy to see that $\phi(\ZZ^n)$ is not dense in $\RR$ if and only if $\phi$ is a multiple of a rational form. The famous Oppenheim conjecture, proved by Margulis \cite{Ma-Op1,Ma-Op2}, states that the same statement holds if $\phi$ is a nondegenerate indefinite quadratic form and $n\ge3$. It follows that in both cases, if $\phi(\ZZ^n_{\ne0})$ has a gap at some number $a$ then $\phi$ is a multiple of a rational form. Moreover, a conjecture from Margulis \cite[Conjecture 8]{Ma00} (see also Cassels and Swinnerton-Dyer \cite[Hypothesis A]{CaSD}) states that if $n\ge3$ and $\phi$ is the product of $n$ linearly independent linear forms such that $\phi(\ZZ^n_{\ne0})$ has a gap at $0$, then $\phi$ is a multiple of a rational polynomial. Although this conjecture remains open, it has been proved by Einsiedler, Katok and Lindenstrauss \cite[Theorem 1.6]{EKL} that in the space of products of $n$ linearly independent linear forms, the set of polynomials $\phi$ with $\phi(\ZZ^n_{\ne0})$ having a gap at $0$ has the same Hausdorff dimension as the set of multiples of rational polynomials, namely $1$.

The situation is completely different when $n=2$: It is proved by Kleinbock and Weiss \cite{KW3} that given any countable subset $A$ of $\RR$, the set of $\phi$ in the space of nondegenerate indefinite binary quadratic forms (or equivalently, products of two linearly independent linear forms) such that $\phi(\ZZ^2_{\ne0})$ has a gap at every $a\in A$ is thick in this space. In this section, we use Theorem \ref{A2} to extend the last result.
\smallskip

To begin with, let us introduce some notation. Let $\widehat{C}(\RR^n)$ denote the set of $\phi\in C(\RR^n)$ such that $\phi(\ZZ^n)$ is not dense in $\phi(\RR^n)$, that is,
$$\widehat{C}(\RR^n):=\left\{\phi\in C(\RR^n):\phi(\RR^n)\not\subset\overline{\phi(\ZZ^n)}\right\}.$$
For $a\in\RR$, let $\widehat{C}_a(\RR^n)$ denote the set of $\phi\in C(\RR^n)$ such that $\phi(\ZZ^n_{\ne0})$ has a gap at $a$, that is,
$$\widehat{C}_a(\RR^n):=\left\{\phi\in C(\RR^n):a\notin \overline{\phi(\ZZ^n_{\ne0})}\right\}.$$
It is easy to see that $\phi\in \widehat{C}(\RR^n)$ if and only if $\phi\in\bigcup_{a\in \phi(\RR^n)}\widehat{C}_a(\RR^n)$.
We would like to understand the sets $\widehat{C}(\RR^n)$ and $\widehat{C}_a(\RR^n)$. However, they are too large to be addressed. To proceed, consider the natural right action of $\SL_n(\RR)$ on $C(\RR^n)$, which is given by
$$\SL_n(\RR)\times C(\RR^n)\to C(\RR^n), \qquad (g,\phi)\mt \phi\circ g.$$
Here $\RR^n$ is understood as the space of column vectors, and $g\in\SL_n(\RR)$ is identified with the left multiplication by $g$ on $\RR^n$. For $\phi\in C(\RR^n)$, recall the definition \equ{orbit} of the $\SL_n(\RR)$-orbit $\OO(\phi)\cong \Aut(\phi)\backslash\SL_n(\RR)$ of $\phi$, where
$\Aut(\phi)\subset\SL_n(\RR)$ is the stabilizer of $\phi$ in $\SL_n(\RR)$.  $\Aut(\phi)$ is a closed subgroup\footnote{It is easy to see from the Taylor expansion that if $\phi$ is real analytic, then $\Aut(\phi)$ is algebraic. However, even if $\phi$ is smooth, $\Aut(\phi)$ may fail to be algebraic. For example, the function $\phi$ on $\RR^3$ given by
$$\phi(x,y,z)=\begin{cases}
(2xyz+y^2z\log|z|)\exp(-1/y^2|z|), & yz\ne0;\\
0, & yz=0
\end{cases}$$
is smooth, but
$\Aut(\phi)=\left\{\begin{pmatrix}
 \pm e^t & \pm te^t & 0\\
  0 & \pm e^t & 0\\
  0 & 0 & e^{-2t}
\end{pmatrix}:t\in\RR\right\}$ is not algebraic.} of $\SL_n(\RR)$, and hence $\OO(\phi)$ has a natural smooth manifold structure. Let
$$\OOO(\phi):=\OO(\phi)\cap \widehat{C}(\RR^n), \qquad \OOO_a(\phi):=\OO(\phi)\cap \widehat{C}_a(\RR^n), \qquad a\in\RR.$$
Then
$$\OOO(\phi)=\bigcup_{a\in \phi(\RR^n)}\OOO_a(\phi).$$
More generally, for a subset $A$ of $\RR$, denote
$$\OOO_A(\phi)=\bigcap_{a\in A}\OOO_a(\phi)=\left\{\psi\in \OO(\phi):\overline{\psi(\ZZ^n_{\ne0})}\cap A=\emptyset\right\}.$$
Our aim is to understand $\OOO(\phi)$ and $\OOO_A(\phi)$ as subsets of the manifold $\OO(\phi)$.

\smallskip
First, let us observe the following fact.

\begin{proposition}\label{prop-zero-measure}
If $\Aut(\phi)$ is noncompact, then $\OOO(\phi)$ has measure zero (with respect to any smooth measure on $\OO(\phi)$).
\end{proposition}

\begin{proof}
Let \eq{rho}{\rho:\SL_n(\RR)\to\OO(\phi), \qquad g\mt \phi\circ g} be the natural projection. It suffices to prove that the set
\begin{equation}\label{zero-measure}
\rho^{-1}\big(\OOO(\phi)\big)=\left\{g\in\SL_n(\RR):\phi\circ g\in\OOO(\phi)\right\}
\end{equation}
has measure zero with respect to the Haar measure on $\SL_n(\RR)$. Since the group $\Aut(\phi)$ is noncompact, it follows from Moore's Ergodicity Theorem that the $\Aut(\phi)$-action on $\SL_n(\RR)/\SL_n(\ZZ)$ is ergodic. Hence, almost every point in $\SL_n(\RR)/\SL_n(\ZZ)$ has a dense $\Aut(\phi)$-orbit. This implies that for almost every $g\in\SL_n(\RR)$, the set $\Aut(\phi)g\SL_n(\ZZ)$ is dense in $\SL_n(\RR)$. For such a $g$, we have
\begin{align*}
\overline{\phi\circ g(\ZZ^n)}&=\overline{\phi(\Aut(\phi)g\SL_n(\ZZ)\ZZ^n)}\\
&\supset \phi\left(\overline{\Aut(\phi)g\SL_n(\ZZ)}\ZZ^n\right)\\
&=\phi\big(\SL_n(\RR)\ZZ^n\big)\\
&=\phi\circ g(\RR^n),
\end{align*}
that is, $\phi\circ g\notin\OOO(\phi)$. Hence the set \eqref{zero-measure} has measure zero in $\SL_n(\RR)$. This completes the proof.
\end{proof}

In view of Proposition \ref{prop-zero-measure}, it is natural to ask what is the Hausdorff dimension of $\OOO(\phi)$ or $\OOO_A(\phi)$. Let us first review the cases mentioned in the beginning of this section.
\begin{itemize}
\item[(1)] If $\phi\in(\RR^n)^*\smz$, then $\OO(\phi)=(\RR^n)^*\smz$, $\Aut(\phi)$ is conjugate to the group of matrices in $\SL_n(\RR)$ with $(1,0,\ldots,0)$ as the first row, and $\OOO(\phi)$ consists of nonzero multiples of rational linear forms.
  \item[(2)] If $n=p+q\ge3$, where $p,q\ge1$, and $\phi$ is a quadratic form of signature $(p,q)$, then $\OO(\phi)$ is the space of all such forms with the same determinant as $\phi$, $\Aut(\phi)$ is conjugate to $\mathrm{SO}(p,q)$, and by the Oppenheim conjecture (Margulis' theorem), the set $\OOO(\phi)$ consists of forms in $\OO(\phi)$ that are multiples of rational forms. It follows that
      $$\dim  \OOO_a(\phi)=\dim  \OOO(\phi)=0, \qquad \forall \, a\in\RR.$$
  \item[(3)] If $n\ge3$ and $\phi$ is the product of $n$ linearly independent linear forms, then $\OO(\phi)$ is the space of all such polynomials with the same ``determinant" as $\phi$, and the identity component of $\Aut(\phi)$ is conjugate to the group of positive diagonal matrices in $\SL_n(\RR)$. The above-mentioned conjecture from \cite{Ma00,CaSD} states that every polynomial in $\OOO_0(\phi)$ is a multiple of a rational polynomial; it is proved in \cite{EKL} that $$\dim \OOO_0(\phi)=0.$$
  \item[(4)] If $n=2$ and $\phi$ is a nondegenerate indefinite binary quadratic form (or equivalently, the product of two linearly independent linear forms), then $\OO(\phi)$ is the space of all such forms with the same determinant as $\phi$, and the identity component of $\Aut(\phi)$ is a one-parameter diagonalizable subgroup of $\SL_2(\RR)$. It is proved in \cite{KW3} that for any countable subset $A$ of $\RR$, the set $\OOO_A(\phi)$ is thick in $\OO(\phi)$. In particular, we have $$\dim  \OOO_A(\phi)=\dim\OO(\phi)=2.$$
\end{itemize}

\subsection{A sufficient condition for the winning property of $\OOO_A(\phi)$} \label{suffcond}

In this section, we prove a general theorem which extends Case (4) above.
Let $F = \{g_t : t\in\RR\}$ be a one-parameter subgroup of $\SL_n(\RR)$. Say that $F$ is {\sl non-quasiunipotent} if $\rho(\Ad g_1) > 1$. Since $\SL_n(\RR)$ is unimodular, this is equivalent to $\rho(\Ad g_{-1}) > 1$, and hence to the subgroups $G_{F^+}^{\max}$ and  $G_{F^-}^{\max}$ (the  maximally expanding horospherical subgroups of $G$ relative to $g_1$ and $g_{-1}$ respectively) being nontrivial.

\smallskip
Let $X_n$ denote the space of unimodular lattices in $\RR^n$, which is identified with the space $\SL_n(\RR)/\SL_n(\ZZ)$ in the natural way. Let us first recall the following conjecture\footnote{The conjecture is stated in \cite[Conjecture 7.1]{AGK} for any Lie group $G$ and any lattice $\Gamma\subset G$. In this case, it is proved in \cite{KM} that the set $E(F,\infty)$ is thick if and only if $F$ has the so-called property (Q). If $F$ is Ad-diagonalizable, then it has property (Q).
So in \cite[Conjecture 7.1]{AGK}, $F$ is assumed to be Ad-diagonalizable for simplicity. For $G=\SL_n(\RR)$, any non-quasiunipotent $F$ has property (Q). We prefer to state Conjecture \ref{C:AGK} for any non-quasiunipotent $F$.} from \cite{AGK}:

\begin{conjecture}\label{C:AGK}
Let $F$ be a non-quasiunipotent one-parameter subgroup of $\SL_n(\RR)$. Then the set
$$E(F,\infty):=\{\Lambda\in X_n:F\Lambda \text{ is bounded}\}$$
is HAW in $X_n$.
\end{conjecture}

Conjecture \ref{C:AGK} is proved for $n=2$ in \cite{KW3}, and for $n=3$ and diagonalizable $F$ in \cite{AGK}. It also follows from a result in \cite{BFS} that the conjecture holds for diagonalizable $F$ such that $g_1$ has only two eigenvalues (see Theorem \ref{T:conj} below). Moreover, the conjecture is proved in \cite{GW} for diagonalizable $F$ such that the eigenvalues $\lambda_1,\ldots,\lambda_n$ of $g_1$ satisfy
$$\#\{i:|\lambda_i|<1\}=1 \qquad \text{ and } \qquad \#\left\{i:|\lambda_i|=\max_{1\le j\le n}|\lambda_j|\right\}\ge n-2.$$
See also \cite{AGGL,Da2,KM,KW1} for other related results.

\smallskip

The main result of this section, which is a generalization of Theorem \ref{C}, is as follows.

\begin{theorem}\label{T:app1-main}
Let $\phi\in C(\RR^n)$. Suppose that $\Aut(\phi)$ has a one-parameter non-quasiunipotent subgroup $F=\{g_t:t\in\RR\}$ satisfying the following conditions:
\begin{itemize}
  \item[(i)] There exists a continuous function $N:\RR\to[0,\infty)$ with $N(0)=0$ such that
$$N\big(\phi(\pv)-\phi(0)\big)\ge\dist(F\pmb{v},0), \qquad \forall \, \pmb{v}\in\RR^n.$$
\item[(ii)] For any real number $a\ne \phi(0)$, the set $\phi^{-1}(a)$ is contained in a countable union of $F$-invariant $C^1$ submanifolds of $\RR^n$ that are both $G_{F^+}^{\max}$-transversal and $G_{F^-}^{\max}$-transversal\footnote{Here the transversality is understood in the sense of the linear action of $\SL_n(\RR)$ on $\RR^n\smz$.}.
\item[(iii)] Conjecture \ref{C:AGK} holds for $F$.
\end{itemize}
Then, for any countable subset $A$ of $\RR$, the set $\OOO_A(\phi)$ is HAW on $\OO(\phi)$.
\end{theorem}

Note that condition (i) in Theorem \ref{T:app1-main} is independent of the choice of the norm $\|\cdot\|$ on $\RR^n$ that is used to define $\dist(F\pv,0):=\inf_{t\in\RR}\|g_t\pv\|$.

\smallskip

We first deduce Theorem \ref{T:app1-main} from the following dual statement.

\begin{theorem}\label{T:app1-dual}
Let $\phi\in C(\RR^n)$ be such that $\Aut(\phi)$ has a one-parameter non-quasiunipotent subgroup $F$ satisfying conditions {\rm (i)--(iii)} in Theorem \ref{T:app1-main}. Then for every $a\in\RR$, the set
$$X(\phi,a):=\left\{\Lambda\in X_n:a\notin\overline{\phi(\Lambda\smz)}\right\}$$
is HAW on $X_n$.
\end{theorem}

\begin{proof}[Proof of Theorem \ref{T:app1-main} assuming Theorem \ref{T:app1-dual}]
Note that $\OOO_A(\phi)$ is the image of the set
\begin{equation}\label{group-HAW-1}
\{g\in\SL_n(\RR):\phi\circ g\in\OOO_A(\phi)\}
\end{equation}
under the projection $\rho$ as in \equ{rho}.
By Lemma \ref{L:haw-proj}, it suffices to show that the set \eqref{group-HAW-1} is HAW on $\SL_n(\RR)$. Let \eq{pi}{\pi:\SL_n(\RR)\to X_n, \qquad \pi(g)=g\ZZ^n} be the natural projection. The set \eqref{group-HAW-1} is equal to
$$\left\{g\in\SL_n(\RR):\overline{\phi(g\ZZ^n_{\ne0})}\cap A=\emptyset\right\}=\bigcap_{a\in A}\pi^{-1}\big(X(\phi,a)\big).$$
Assuming Theorem \ref{T:app1-dual}, each $X(\phi,a)$ is HAW on $X_n$. Thus, the set \eqref{group-HAW-1} is HAW on $\SL_n(\RR)$, and hence $\OOO_A(\phi)$ is HAW on $\OO(\phi)$.
\end{proof}

In order to prove Theorem \ref{T:app1-dual}, let us introduce the following notation. For $\phi\in C(\RR^n)$ and $a\in\RR$, denote
$$Z_{\phi,a}=\big\{\Lambda\in X_n:a\in \phi\big(\Lambda\smz\big)\big\}.$$
Then let us prove the following lemma, which relates gaps in $\phi(\Lambda\smz)$ to dynamical properties of the orbit $F\Lambda$.

\begin{lemma}\label{L:relation}
Let $\phi\in C(\RR^n)$, and let $F$ be a one-parameter subgroup of $\Aut(\phi)$ satisfying condition {\rm (i)} in Theorem \ref{T:app1-main}. Then
\begin{itemize}
  \item[(1)] $E(F,\infty)\subset X\big(\phi,\phi(0)\big)$.
  \item[(2)] For any $a\ne \phi(0)$, we have
  $$E(F,Z_{\phi,a})\cap E(F,\infty)\subset X(\phi,a)\subset E(F,Z_{\phi,a}).$$
\end{itemize}
\end{lemma}

\begin{proof}
(1) It suffices to prove that if $\Lambda\in X_n$ and $\phi(0)\in\overline{\phi(\Lambda\smz)}$ then $F\Lambda$ is unbounded. Let $\pmb{v}_k\in \Lambda\smz$ be such that $\phi(\pmb{v}_k)\to \phi(0)$. Then, condition (i) in Theorem \ref{T:app1-main} implies that
$$\dist(F\pmb{v}_k,0)\le N\big(\phi(\pmb{v}_k)-\phi(0)\big)\to0.$$
Hence there are $t_k\in\RR$ such that $g_{t_k}\pmb{v}_k\to0$. It then follows from Mahler's criterion that the sequence $g_{t_k}\Lambda$ in $X_n$ is unbounded. Hence $F\Lambda$ is unbounded.

(2) Suppose to the contrary that the first inclusion does not hold. Then there exists $\Lambda\in X_n$ such that $F\Lambda$ is bounded, $\overline{F\Lambda}\cap Z_{\phi,a}=\varnothing$, but $a\in\overline{\phi(\Lambda\smz)}$. Let $\pmb{v}_k\in \Lambda\smz$ be such that $\phi(\pmb{v}_k)\to a$. Since
$$\dist(F\pmb{v}_k,0)\le N\big(\phi(\pmb{v}_k)-\phi(0)\big)\to N\big(a-\phi(0)\big),$$
there exist $t_k\in\RR$ such that $g_{t_k}\pmb{v}_k$ is a bounded sequence in $\RR^n$. Note that the sequence $g_{t_k}\Lambda$ in $X_n$ is also bounded. By passing to subsequences, we may assume that $g_{t_k}\pmb{v}_k\to \pmb{v}\in\RR^n$ and $g_{t_k}\Lambda\to\Delta\in X_n$. It follows that $\pmb{v}\in\Delta$ and
$$\phi(\pmb{v})=\lim_{k\to\infty}\phi(g_{t_k}\pmb{v}_k)=\lim_{k\to\infty}\phi(\pmb{v}_k)=a.$$
Together with the fact that $a\ne \phi(0)$, this also implies that $\pmb{v}\ne0$. So $\Delta\in \overline{F\Lambda}\cap Z_{\phi,a}$, a contradiction.

To prove the second inclusion, it suffices to show that if $\Lambda\in X_n$ and $\overline{F\Lambda}\cap Z_{\phi,a}\ne\varnothing$ then $a\in\overline{\phi(\Lambda\smz)}$. Let $t_k\in\RR$ and $\Delta\in Z_{\phi,a}$ be such that $g_{t_k}\Lambda\to\Delta$, and let $\pmb{v}\in\Delta\smz$ be such that $\phi(\pmb{v})=a$. Then there exist $\pmb{v}_k\in \Lambda\smz$ such that $g_{t_k}\pmb{v}_k\to \pmb{v}$. It follows that
$$\phi(\pmb{v}_k)=\phi(g_{t_k}\pmb{v}_k)\to \phi(\pmb{v})=a.$$
Hence $a\in\overline{\phi(\Lambda\smz)}$.
\end{proof}

In view of Lemma \ref{L:relation}, to prove Theorem \ref{T:app1-dual} we need only to show that $E(F,Z_{\phi,a})$ is HAW for every $a\ne \phi(0)$.
For a subset $M$ of $\RR^n$, denote
$$Z_M:=\{\Lambda\in X_n:P(\Lambda)\cap M\ne\emptyset\},$$
where
$$P(\Lambda):=\{\pmb{v}\in\Lambda:\pmb{v}/k\notin\Lambda \text{ for every integer } k\ge2\}$$
is the set of primitive vectors in $\Lambda$. We first use Theorem \ref{A2} to prove:

\begin{proposition}\label{P:app0}
Let $F$ be a one-parameter non-quasiunipotent subgroup of $\SL_n(\RR)$, and let $M$ be an $F$-invariant $C^1$ submanifold  of $\RR^n$ which is both $G_{F^+}^{\max}$-transversal and $G_{F^-}^{\max}$-transversal. Then $E(F,Z_M)$ is HAW on $X_n$.
\end{proposition}

\begin{proof}
Let $\pi$ be as in \equ{pi}, and let $p_1:\SL_n(\RR)\to\RR^n$ be the map that sends a matrix to its first column. It is easy to see that $Z_M=\pi\big(p_1^{-1}(M)\big)$. Since $p_1$ is a submersion and is $F$-equivariant with respect to the left multiplications on $\SL_n(\RR)$ and $\RR^n$, the set $p_1^{-1}(M)$ is a left $F$-invariant $C^1$ submanifold of $\SL_n(\RR)$. Thus, we can select a countable family $\{Z'_i:i\in\NN\}$ of codimension one $C^1$ submanifolds of $p_1^{-1}(M)$ such that
\begin{itemize}
  \item For any $i\in\NN$ and $g\in Z'_i$, we have $T_g(Fg)\not\subset T_gZ'_i$;
  \item $p_1^{-1}(M)=\bigcup_{i\in\NN}FZ'_i$;
  \item For any $i\in\NN$, there exists an open subset $U_i$ of $\SL_n(\RR)$ containing $Z'_i$ such that $\pi|_{U_i}$ is a diffeomorphism onto $\pi(U_i)$.
\end{itemize}
Let $Z_i=\pi(Z'_i)$. Then $\{Z_i:i\in\NN\}$ is a family of $C^1$ submanifolds of $X_n$, and
$$Z_M=\pi\big(p_1^{-1}(M)\big)=\bigcup_{i\in\NN}\pi(FZ'_i)=\bigcup_{i\in\NN}FZ_i.$$
It follows that
$$E(F,Z_M)=\bigcap_{i\in\NN}E(F,FZ_i)=\bigcap_{i\in\NN}E(F,Z_i).$$
Thus, in view of Theorem \ref{A2}, it suffices to show that each $Z_i$ is $\big(F,G_{F^+}^{\max}\big)$-transversal and $\big(F,G_{F^-}^{\max}\big)$-transversal.

Let $i\in\NN$, $\Lambda\in Z_i$, and let $g\in Z'_i$ be such that $\Lambda=\pi(g)$. Then
$$T_\Lambda(F\Lambda)=(d\pi)_g\big(T_g(Fg)\big)\not\subset (d\pi)_g(T_gZ'_i)=T_\Lambda Z_i.$$
On the other hand, for $H\in\{G_{F^+}^{\max},G_{F^-}^{\max}\}$ we have $T_g(Hg)\not\subset T_g\big(p_1^{-1}(M)\big)$,
for otherwise, if $T_g(Hg)\subset T_g\big(p_1^{-1}(M)\big)$, and if we let $\pmb{v}=p_1(g)$, then
$$T_v(H\pmb{v})=(dp_1)_g\big(T_g(Hg)\big)\subset (dp_1)_g\left(T_g\big(p_1^{-1}(M)\big)\right)=T_{\pmb{v}}M,$$
contrary to the assumption that $M$ is $H$-transversal. Thus,
\begin{align*}
T_\Lambda(H\Lambda)&=(d\pi)_g\big(T_g(Hg)\big)\\
&\not\subset (d\pi)_g\left(T_g\big(p_1^{-1}(M)\big)\right)\\
&=(d\pi)_g\big(T_gZ'_i\oplus T_g(Fg)\big)\\
&=(d\pi)_g(T_gZ'_i)\oplus\big(d\pi)_g(T_g(Fg)\big)\\
&=T_\Lambda Z_i\oplus T_\Lambda(F\Lambda).
\end{align*}
This proves that each $Z_i$ is $\big(F,G_{F^+}^{\max}\big)$-transversal and $\big(F,G_{F^-}^{\max}\big)$-transversal, hence completes the proof of the proposition.
\end{proof}

\begin{remark}
Even if $M$ is a nice submanifold of $\RR^n$, the set $Z_M$ may fail to be a submanifold of $X_n$. In fact, if $\dim M<n$ and $M$ contains at least two linearly independent vectors in $P(\Lambda)$, then $\Lambda$ is a self-intersection point of $Z_M$.
\end{remark}

We now derive the HAW property of $E(F,Z_{\phi,a})$ from the above proposition:

\begin{corollary}\label{C:app1}
Let $\phi\in C(\RR^n)$, and let $F$ be a one-parameter non-quasiunipotent subgroup of $\Aut(\phi)$ satisfying condition {\rm (ii)} in Theorem \ref{T:app1-main}. Then for any $a\ne \phi(0)$, the set $E(F,Z_{\phi,a})$ is HAW on $X_n$.
\end{corollary}

\begin{proof}
Suppose that $\phi^{-1}(a)\subset\bigcup_{i\in\NN}M_i$, where each $M_i$ is an $F$-invariant $C^1$ submanifold of $\RR^n$ and is both $G_{F^+}^{\max}$-transversal and $G_{F^-}^{\max}$-transversal. Since $\Lambda\smz=\bigcup_{k\in\NN}kP(\Lambda)$, we have
\begin{align*}
Z_{\phi,a} & =\{\Lambda\in X_n:(\Lambda\smz)\cap \phi^{-1}(a)\ne\emptyset\}\\
& \subset\bigcup_{k,i\in\NN}\{\Lambda\in X_n:kP(\Lambda)\cap M_i\ne\emptyset\}\\
& =\bigcup_{k,i\in\NN}Z_{\frac1kM_i}.
\end{align*}
This implies that
$$E(F,Z_{\phi,a})\supset\bigcap_{k,i\in\NN}E(F,Z_{\frac1kM_i}).$$
Note that each $\frac1kM_i$ is an $F$-invariant $C^1$ submanifold of $\RR^n$ and is both $G_{F^+}^{\max}$-transversal and $G_{F^-}^{\max}$-transversal. Thus, it follows from Proposition \ref{P:app0} that each $E(F,Z_{\frac1kM_i})$ is HAW. Hence $E(F,Z_{\phi,a})$ is HAW.
\end{proof}

It is now straightforward to derive Theorem \ref{T:app1-dual} from Lemma \ref{L:relation} and Corollary \ref{C:app1}.

\begin{proof}[Proof of Theorem \ref{T:app1-dual}]
If $a=\phi(0)$, then by Lemma \ref{L:relation}(1), the set $X(\phi,a)$ contains $E(F,\infty)$, which is HAW by the assumption. If $a\ne \phi(0)$, then by Lemma \ref{L:relation}(2), the set $X(\phi,a)$ contains $E(F,\infty)\cap E(F,Z_{\phi,a})$, which is HAW by the assumption and Corollary \ref{C:app1}.
\end{proof}

\section{Applications to GIBFs}\label{gibfs}

\subsection{Proof of Theorem \ref{C}}

Recall that in \S\ref{gaps} we defined  a continuous function $\phi: \RR^n \rightarrow \RR$ to be a generalized indefinite binary form, abbreviated (GIBF) if there exists a nontrivial decomposition $\RR^n=U\oplus W$ such that  conditions (IB-1), (IB-2) and (IB-3) hold. Throughout this section, let $\RR^n=U\oplus W$ be such a nontrivial decomposition, let $G = \SL_n(\RR)$, and let $F=\{g_t:t\in\RR\}\subset G$ be as in \eqref{E:gt}. For $\pmb{v}\in\RR^n$, we always let $\pmb{u}$ and $\pmb{w}$ denote the unique vectors with $\pmb{u}\in U$ and $\pmb{w}\in W$ such that $\pmb{v}=\pmb{u}+\pmb{w}$. As a sample case of the decomposition, one can take
\begin{equation}\label{E:sample}
U=\RR\pmb{e}_1 \oplus\cdots \oplus\RR\pmb{e}_p, \qquad W=\RR\pmb{e}_{p+1} \oplus\cdots \oplus\RR\pmb{e}_n,
\end{equation}
where $\{\pmb{e}_1,\dots,\pmb{e}_n\}$ is the standard basis of $\RR^n$.
In this case   \eqref{E:gt} reduces to \eqref{E:gt-st}, see Example \ref{example0}.

\smallskip

First, let us observe the following facts.

\begin{lemma}\label{L:cor}
\begin{itemize}
  \item[(1)] For any $\pmb{v}\in \RR^n$, we have $\dist(F\pmb{v},0)\le 2\|\pmb{u}\|^{p/n}\|\pmb{w}\|^{q/n}$.
\item[(2)] Suppose $\pmb{v}\in \RR^n\sm(U\cup W)$. Then under the natural identification $T_{\pmb{v}}\RR^n\cong\RR^n$, we have $T_{\pmb{v}}(G_{F^+}^{\max}\pmb{v})=U$, $T_{\pmb{v}}(G_{F^-}^{\max}\pmb{v})=W$.
\item[(3)] Suppose $\phi\in C(\RR^n)$ satisfies conditions {\rm (IB-1)} and {\rm (IB-2)}. Then $\phi^{-1}(0)=U\cup W$.
\end{itemize}
\end{lemma}

\begin{proof}
(1) If $\pv\in U\cup W$, then both sides of the inequality are equal to $0$. Suppose $\pv\notin U\cup W$. Then there exists $t_0\in\RR$ such that $e^{t_0/p}\|\pu\|=e^{-t_0/q}\|\pw\|=\|\pu\|^{p/n}\|\pw\|^{q/n}$. It follows that
$$\dist(F\pv,0)\le \|g_{t_0}\pv\|=\|e^{t_0/p}\pu+e^{-t_0/q}\pw\|\le e^{t_0/p}\|\pu\|+e^{-t_0/q}\|\pw\|=2\|\pu\|^{p/n}\|\pw\|^{q/n}.$$

(2) Without loss of generality, we may assume that $U$ and $W$ are as in \eqref{E:sample}, and $F$ is as in \eqref{E:gt-st}. Then $\Lg_{F^+}^{\max}$ and  $\Lg_{F^-}^{\max}$ are given by \eqref{E:mainexample}.
Write ${\pmb{v}}=\begin{pmatrix}
  {\pmb{u}} \\ {\pmb{w}}
\end{pmatrix}$, where ${\pmb{u}}\in\RR^p$, ${\pmb{w}}\in\RR^q$. Then ${\pmb{u}}$ and ${\pmb{w}}$ are nonzero. It follows that
$$T_{\pmb{v}}(G_{F^+}^{\max}{\pmb{v}})=\Lg_{F^+}^{\max}{\pmb{v}}=\left\{\begin{pmatrix}
  A{\pmb{w}} \\ 0
\end{pmatrix}:A\in\mathrm{M}_{p\times q}(\RR)\right\}=U,$$
$$T_{\pmb{v}}(G_{F^-}^{\max}{\pmb{v}})=\Lg_{F^-}^{\max}{\pmb{v}}=\left\{\begin{pmatrix}
  0 \\ B{\pmb{u}}
\end{pmatrix}:B\in\mathrm{M}_{q\times p}(\RR)\right\}=W.$$
This proves (2).
\smallskip

(3) Suppose ${\pmb{v}}\in U\cup W$. Then there exist $t_k\in\RR$ such that $g_{t_k}{\pmb{v}}\to0$. It follows from condition (IB-1) that $\phi({\pmb{v}})=\phi(g_{t_k}{\pmb{v}})\to \phi(0)=0$. Hence $\phi({\pmb{v}})=0$. Conversely, if $\phi({\pmb{v}})=0$, then condition (IB-2) implies $\|{\pmb{u}}\|^p\|{\pmb{w}}\|^q=0$, which means that ${\pmb{v}}\in U\cup W$.
\end{proof}

Next, let us notice the following result, which can be easily deduced from one of the main results of \cite{BFS}.

\begin{theorem}\label{T:conj}
Conjecture \ref{C:AGK} holds for $F$ as in \eqref{E:gt}.
\end{theorem}

\begin{proof}
Without loss of generality, we may assume $F$ is as in \eqref{E:gt-st}. Let $F^+=\{g_t:t\ge0\}$, $F^-=\{g_t:t\le0\}$, and let
$E(F^\pm,\infty)$ be the set of $\Lambda\in X_n$ such that $F^\pm\Lambda$ is bounded. As is well known, bounded $F^+$-orbits in $X_n$ are related to badly approximable matrices. More precisely, it is shown in \cite{Da1} that a matrix $A\in\mathrm{M}_{p\times q}(\RR)$ is badly approximable if and only if the orbit $F^+\begin{pmatrix}
  I_p & A \\ 0 & I_q
\end{pmatrix}\ZZ^n$ is bounded. On the other hand, it is proved in \cite{BFS} that the set of badly approximable matrices is HAW. Starting from these results, and using the method as in the proof of \cite[Theorem 1.2]{AGK}, it is easy to show that $E(F^+,\infty)$ is HAW. Let $\varphi$ be the diffeomorphism of $X_n$ given by $\varphi(g\Gamma)=(g^\T)^{-1}\Gamma$. Then $E(F^-,\infty)=\varphi\big(E(F^+,\infty)\big)$, hence is also HAW. Therefore, $E(F,\infty)=E(F^+,\infty)\cap E(F^-,\infty)$ is HAW.
\end{proof}

It is now straightforward to deduce Theorem \ref{C} from Theorem \ref{T:app1-main}.

\begin{proof}[Proof of Theorem \ref{C}]
Since $\phi$ is a GIBF, the group $F$ given by \eqref{E:gt} is a one-parameter non-quasiunipotent subgroup of $\Aut(\phi)$. It follows from conditions (IB-2), (IB-3) and Lemma \ref{L:cor} that conditions (i) and (ii) in Theorem \ref{T:app1-main} are satisfied.
Moreover, condition (iii) in Theorem \ref{T:app1-main} follows from Theorem \ref{T:conj}. Thus Theorem \ref{T:app1-main} implies the conclusion of Theorem \ref{C}.
\end{proof}

\subsection{Examples}\label{Examples}

In this subsection, we give several interesting examples of GIBFs. Let us first notice the following fact, which will be used to verify condition (IB-3).

\begin{lemma}\label{L:IB3}
Let $\RR^n=U\oplus W$ be a nontrivial decomposition, $F$ be as in \eqref{E:gt}, and $\phi_1,\ldots,\phi_m\in C(\RR^n)$ be finitely many $F$-invariant functions satisfying
\begin{equation}\label{E:homo}
\begin{aligned}
&  \text{$\phi_i(0)=0$, $\phi_i$ is $C^1$ on $\RR^n\sm \phi_i^{-1}(0)$, and} \\
&  \text{$\DT \phi_i(t{\pmb{u}}+{\pmb{w}})\ne0$ for any ${\pmb{v}}\in\RR^n\sm \phi_i^{-1}(0)$.}
\end{aligned}
\end{equation}
Then the function
\begin{equation}\label{E:max}
\phi({\pmb{v}}):=\max_{1\le i\le m}\phi_i({\pmb{v}})
\end{equation}
satisfies {\rm (IB-3)}.
\end{lemma}

\begin{proof}
Note that $\phi^{-1}(a)\subset\bigcup_{i=1}^m\phi_i^{-1}(a)$. Thus, it suffices to show that if $a\ne0$, then each $\phi_i^{-1}(a)$ is an $F$-invariant $C^1$ submanifold of $\RR^n$ that is both $U$-transversal and $W$-transversal.
Since $\phi_i$ is $F$-invariant, so is the set $\phi_i^{-1}(a)$. It follows from \eqref{E:homo} that $\phi_i$ is a $C^1$ submersion on $\RR^n\sm \phi_i^{-1}(0)$. So $\phi_i^{-1}(a)$ is a $C^1$ submanifold of $\RR^n$, and for ${\pmb{v}}\in\phi_i^{-1}(a)$, we have $T_{\pmb{v}}\big(\phi_i^{-1}(a)\big)=\ker(d\phi_i)_{\pmb{v}}$. By \eqref{E:homo} again, we have
$$(d\phi_i)_{\pmb{v}}({\pmb{u}})=\DT \phi_i(t{\pmb{u}}+{\pmb{w}})\ne0.$$
This means that ${\pmb{u}}\notin \ker(d\phi_i)_{\pmb{v}}$, which implies that $U\not\subset T_{\pmb{v}}\big(\phi_i^{-1}(a)\big)$, that is, $\phi_i^{-1}(a)$ is $U$-transversal. On the other hand, it follows from
$$0=\DDT \phi_i(g_t{\pmb{v}})=(d\phi_i)_{\pmb{v}}(\DDT g_t{\pmb{v}})=(d\phi_i)_{\pmb{v}}({\pmb{u}}/p-{\pmb{w}}/q)$$
that $(d\phi_i)_{\pmb{v}}({\pmb{w}})\ne0$. So ${\pmb{w}}\notin \ker(d\phi_i)_{\pmb{v}}$, which implies that $\phi_i^{-1}(a)$ is $W$-transversal. The proof of the lemma is thus completed.
\end{proof}

A special case of Lemma \ref{L:IB3} is that if $\phi\in C(\RR^n)$ satisfies (IB-1), (IB-2) and \eqref{E:homo}, then it is a GIBF. We use this special case to verify Examples \ref{example1}--\ref{example4} below.

\begin{example}\label{example1}
Let $\RR^n=U\oplus W$ be a nontrivial decomposition, let $\|\cdot\|$ be a norm on $\RR^n$ that is $C^1$ on $\RR^n\smallsetminus(U\cup W)$, and consider the function
\begin{equation}\label{E:norm}
\phi({\pmb{u}}+{\pmb{w}})=\|{\pmb{u}}\|^p\|{\pmb{w}}\|^q.
\end{equation}
Conditions (IB-1) and (IB-2) are clearly satisfied. Also, for ${\pmb{u}}\in U\smz$ and ${\pmb{w}}\in W\smz$ we have
$$\DT \phi(t{\pmb{u}}+{\pmb{w}})=\DT t^p\|{\pmb{u}}\|^p\|{\pmb{w}}\|^q=p\|{\pmb{u}}\|^p\|{\pmb{w}}\|^q\ne0,$$
which implies that \eqref{E:homo} is satisfied. Thus $\phi$ is a GIBF. Note that the polynomial \eqref{E:poly2} is of the form \eqref{E:norm}, hence is a GIBF.
\end{example}

\begin{example}\label{example2}
Let $n=2p$ be even, $\varepsilon\ge0$, and consider the polynomial
$$\phi_\varepsilon(x_1,\ldots,x_n)=\left(\sum_{i=1}^px_ix_{p+i}\right)^2+\varepsilon\left(\sum_{i=1}^px_i^2\right)\left(\sum_{i=1}^px_{p+i}^2\right).$$
Let us verify that if $\varepsilon>0$ then $\phi_\varepsilon$ is a GIBF. Let $U$ and $W$ be as in the sample case \eqref{E:sample} with $q=p$. Then (IB-1) is clear, (IB-2) is satisfied for $N(\lambda)=|\lambda/\varepsilon|^{p/2}$ and the standard Euclidean norm, and \eqref{E:homo} is also satisfied as
$$\DT\phi_\varepsilon(tx_1,\ldots,tx_p,x_{p+1},\ldots,x_n)=2\phi_\varepsilon(x_1,\ldots,x_n).$$
Thus $\phi_\varepsilon$ is a GIBF, and hence the set $\OOO_A(\phi_\varepsilon)$ is HAW for any countable $A\subset\RR$. (The same argument also shows that the polynomial \eqref{E:poly3} is a GIBF.) However, if $\varepsilon=0$ and $n\ne2$, then $\phi_0$ is the square of a quadratic form of signature $(p,p)$, and the Oppenheim conjecture (Margulis' theorem) implies that $\dim  \OOO(\phi_0)=0$.
\end{example}

\begin{example}\label{example3}
The polynomial \eqref{E:poly4}, namely, the function $\phi$ on $\RR^3$ given by
$$\phi(x_1,x_2,x_3)=x_1x_2^2+x_1^3x_3^6$$
is a GIBF. In fact, let $U$ and $W$ be as in \eqref{E:sample} with $p=1$ and $q=2$, then (IB-1) is clear, (IB-2) is satisfied for $N(\lambda)=\max\{|\lambda|,|\lambda|^{\frac13}\}$ and the supremum norm as
$$N\big(\phi(x_1,x_2,x_3)\big)=\max\left\{|x_1x_2^2|+|x_1x_3^2|^3,\left(|x_1x_2^2|+|x_1x_3^2|^3\right)^{\frac13}\right\}\ge\max\big\{|x_1x_2^2|,|x_1x_3^2|\big\},$$
and \eqref{E:homo} is also satisfied as
$$\textstyle \DT \phi(tx_1,x_2,x_3)=x_1x_2^2+3x_1^3x_3^6\ne0$$
if $x_1\ne0$ and $(x_2,x_3)\ne(0,0)$.
\end{example}

\begin{example}\label{example4}
The function $\phi$ on $\RR^4$ given by
$$\phi(x,y,z,s)=x^2z^2+\exp(y^2s^2)+\log(1+x^2s^2+y^2z^2)-1$$
is a GIBF. In fact, let $U$ and $W$ be as in \eqref{E:sample} with $p=q=2$, then (IB-1) is clear, (IB-2) is satisfied for $N(\lambda)=e^{|\lambda|}-1$ and the supremum norm as
$$\phi(x,y,z,s)\ge\log\left(1+\max\{x^2,y^2\}\max\{z^2,s^2\}\right),$$
and \eqref{E:homo} is also satisfied as
$$\textstyle \DT \phi(tx,ty,z,s)=2x^2z^2+2y^2s^2\exp(y^2s^2)+\frac{2(x^2s^2+y^2z^2)}{1+x^2s^2+y^2z^2}>0$$
if $(x,y)\ne(0,0)$ and $(z,s)\ne(0,0)$.
\end{example}

The function $\phi$ in the next example can be written in the form \eqref{E:max}.

\begin{example}\label{example5}
Let $p,q\ge1$ be such that $p+q=n$, and let
$$\phi(x_1,\ldots,x_n)=\max\{|x_1|,\ldots,|x_p|\}^p\max\{|x_{p+1}|,\ldots,|x_n|\}^q.$$
It is easy to see that (IB-1) and (IB-2) are satisfied for $U$ and $W$ be as in \eqref{E:sample}. To verify (IB-3), let us write
$$\phi=\max_{1\le i\le p, 1\le j\le q}\phi_{ij},$$
where
$$\phi_{ij}(x_1,\ldots,x_n)=|x_i|^p|x_{p+j}|^q.$$
Then each $\phi_{ij}$ satisfies \eqref{E:homo}. By Lemma \ref{L:IB3}, $\phi$ satisfies (IB-3), and thus is a GIBF.
\end{example}

We conclude this section by a example that is not covered by Lemma \ref{L:IB3}.

\begin{example}\label{example6}
Let $r>0$. We verify that the function
\begin{equation}\label{E:example6}
\phi(x_1,\ldots,x_n)=x_1\left(\sum_{i=2}^n|x_i|^r\right)^{\frac{n-1}{r}}
\end{equation}
is a GIBF. Let $U$ and $W$ be as in \eqref{E:sample} with $p=1$ and $q=n-1$. Then (IB-1) is clear, and (IB-2) is satisfied for $N(\lambda)=|\lambda|$ and the supremum norm on $\RR^n$. To verify (IB-3), for a subset $I$ of $\{2,\ldots,n\}$ we denote
$$V_I=\{(x_1,\ldots,x_n)\in\RR^n:x_i\ne0 \text{ for every $i\in I$, and $x_j=0$ for every $j\in\{2,\ldots,n\}\sm I$}\}.$$
Then $\big\{V_I:I\subset\{2,\ldots,n\}\big\}$ is a partition of $\RR^n$, and thus for any $a\in\RR$ we have
$$\phi^{-1}(a)=\bigcup_{I\subset\{2,\ldots,n\}}\phi^{-1}(a)\cap V_I.$$
It is straightforward to show that if $a\ne0$, then each $\phi^{-1}(a)\cap V_I$ is an $F$-invariant $C^1$ submanifold of $\RR^n$ and is both $U$-transversal and $W$-transversal. So (IB-3) is satisfied. Hence $\phi$ is a GIBF.
Note that the polynomial \eqref{E:poly1} is the $r=n-1$ case of \eqref{E:example6}. Note also that when $r>1$, one can also verify (IB-3) by verifying \eqref{E:homo}.
\end{example}

\appendix

\section{Proof of Lemma \ref{L:haw-proj}}\label{S:A}

We give the proof of Lemma \ref{L:haw-proj} here. In view of the local nature of the HAW property and the local normal form of a submersion, it suffices to prove the following statement:
\begin{itemize}
\item[($*$)] \textit{Let $\beta\in(0,\frac13)$, $\tilde{\beta}=\beta^2/6$, $V$ a Euclidean space, $W\subset V$ a linear subspace, $P_W:V\to W$ the orthogonal projection, $U\subset V$ an open subset, and $S\subset U$ a subset that is $\tilde{\beta}$-HAW on $U$. Then $P_W(S)$ is $\beta$-HAW on $P_W(U)$.}
\end{itemize}

For the sake of convenience, let us introduce the following concept: We say that two closed balls $B\subset W$ and $\tilde{B}\subset V$ are \textsl{compatible} if $P_W$ sends the center of $\tilde{B}$ to the center of $B$, and the radius of $\tilde{B}$ is twice the radius of $B$. Let us first prove the following lemma.

\begin{lemma}\label{L:submersion}
Let $B\subset W$ and $\tilde{B}\subset V$ be compatible closed balls, and let $\tilde{L}$ be an affine hyperplane in $V$. Let $r$ denote the radius of $B$. Then there exists an affine hyperplane $L=L(B,\tilde{B},\tilde{L})$ in $W$ such that any closed ball in $B\sm L^{(\beta r)}$ of radius $\le \beta r/6$ is compatible with some closed ball in $\tilde{B}\sm\tilde{L}^{(2\tilde{\beta}r)}$.
\end{lemma}

\begin{proof}
Without loss of generality, we may assume that both $B$ and $\tilde{B}$ are centered at the origin. Let $u\in V$ be a unit normal vector of $\tilde{L}$. We divide the proof into two cases.
\smallskip

(1) Suppose $\|P_Wu\|\le1/\sqrt{2}$. We show that any hyperplane $L$ in $W$ has the required property. Let $B'\subset B$ be a closed ball with center $w\in W$ and radius $r'\le \beta r/6$. Let $v_\pm=w\pm r\frac{u-P_Wu}{\|u-P_Wu\|}$. Without loss of generality, assume $\dist(v_+,\tilde{L})\ge\dist(v_-,\tilde{L})$. Let $\tilde{B}'$ be the closed ball in $V$ with center $v_+$ and radius $2r'$. Then $\tilde{B}'$ is compatible with $B'$. We claim that $\tilde{B}'\subset\tilde{B}\sm\tilde{L}^{(2\tilde{\beta}r)}$. First, we have
\begin{align*}
\dist(v_+,\tilde{L})&\ge\frac12\big(\dist(v_+,\tilde{L})+\dist(v_-,\tilde{L})\big)\ge\frac12|\la v_+-v_-,u\ra|\\
&=r\|u-P_Wu\|\ge r/\sqrt{2}\ge2r'+2\tilde{\beta}r.
\end{align*}
This means that $\tilde{B}'\cap\tilde{L}^{(2\tilde{\beta}r)}=\varnothing$. On the other hand, for $v\in\tilde{B}'$ we have
$$\|v\|\le\|v-v_+\|+\|v_+\|\le2r'+\sqrt{2}r\le2r.$$
So $\tilde{B}'\subset\tilde{B}$. This verifies the claim.
\smallskip

(2) Suppose $\|P_Wu\|>1/\sqrt{2}$. We show that the hyperplane $L:=\tilde{L}\cap W$ in $W$ has the required property. Let $B'\subset B\sm L^{(\beta r)}$ be a closed ball with center $w\in W$ and radius $r'\le \beta r/6$. Let $\tilde{B}'$ be the closed ball in $V$ with center $w$ and radius $2r'$. Then $\tilde{B}'$ is compatible with $B'$. We have
$$\dist(w,\tilde{L})=\|P_Wu\|\dist(w,L)\ge(r'+\beta r)/\sqrt{2}\ge2r'+2\tilde{\beta}r,$$
and for $v\in\tilde{B}'$,
$$\|v\|\le\|v-w\|+\|w\|\le2r'+r\le2r.$$
So $\tilde{B}'\subset\tilde{B}\sm\tilde{L}^{(2\tilde{\beta}r)}$. This proves the lemma.
\end{proof}

We now proceed to prove Statement ($*$). For simplicity, let us refer the $\beta$-hyperplane absolute game on $P_W(U)$ with target set $P_W(S)$ as Game 1, and refer the $\tilde{\beta}$-hyperplane absolute game on $U$ with target set $S$ as Game 2. We will construct a winning strategy for Game 1 using the winning strategy for Game 2.

In order to win Game 1, Alice invites two assistants, say Alice's sister and Bob's brother, to play Game 2. Bob's brother will play following Alice's instructions, and Alice's sister will play according to the winning strategy for Game 2. Suppose Bob starts Game 1 by choosing a closed ball $B_0\subset P_W(U)$. Without loss of generality, we may assume that Bob will choose the closed balls $B_i$ so that their radii $r_i$ tend to zero. Let $i_0\ge0$ be the smallest index such that $B_{i_0}$ is compatible with some closed ball in $U$. If $i_0\ne0$, Alice chooses the hyperplane neighborhoods $\{L_i^{(r'_i)}:0\le i<i_0\}$ arbitrarily. After the ball $B_{i_0}$ is chosen by Bob, Alice asks Bob's brother to start Game 2 by choosing a closed ball $\tilde{B}_0\subset U$ compatible with $B_{i_0}$, and next asks her sister to choose a hyperplane neighborhood $\tilde{L}_0^{(\tilde{r}'_0)}\subset V$ according to the winning strategy for Game 2, where $\tilde{r}'_0\le\tilde{\beta}\tilde{r}_0$ and $\tilde{r}_0$ is the radius of $\tilde{B}_0$. Then Alice chooses the hyperplane neighborhood $L_{i_0}^{(r'_{i_0})}\subset W$, where $r'_{i_0}=\beta r_{i_0}$,  $L_{i_0}=L(B_{i_0},\tilde{B}_0,\tilde{L}_0)$ and $L(\cdot,\cdot,\cdot)$ is the function given in Lemma \ref{L:submersion}.
\smallskip

Assume that for some $k\ge0$ and some $i_k\ge k$, the following data have been chosen:
\begin{itemize}
  \item A closed ball $B_{i_k}$ in $W$ chosen by Bob;
  \item A closed ball $\tilde{B}_k$ in $V$ of radius $\tilde{r}_k$ chosen by Bob's brother, which is compatible with $B_{i_k}$;
  \item A hyperplane neighborhood $\tilde{L}_k^{(\tilde{r}'_k)}$ ($\tilde{r}'_k\le\tilde{\beta}\tilde{r}_k$) in $V$ chosen by Alice's sister, according to the winning strategy for Game 2;
  \item A hyperplane neighborhood $L_{i_k}^{(r'_{i_k})}$ in $W$ chosen by Alice, such that $r'_{i_k}=\beta r_{i_k}$ and $L_{i_k}=L(B_{i_k},\tilde{B}_k,\tilde{L}_k)$.
\end{itemize}
(Note that these data have been chosen for $k=0$.) Let $i_{k+1}\ge i_k+1$ be the smallest index such that the radius of the closed ball $B_{i_{k+1}}$ chosen by Bob satisfies $r_{i_{k+1}}\le\beta r_{i_k}/6$. Alice chooses the hyperplane neighborhoods $\{L_i^{(r'_i)}:i_k<i<i_{k+1}\}$ arbitrarily, and then asks Bob's brother to choose a closed ball $\tilde{B}_{k+1}\subset \tilde{B}_k\sm \tilde{L}_k^{(\tilde{\beta}\tilde{r}_k)}$ compatible with $B_{i_{k+1}}$. Note that since $B_{i_{k+1}}\subset B_{i_k+1}\subset B_{i_k}\sm L_{i_k}^{(\beta r_{i_k})}$, the choice of $L_{i_k}$ guarantees that the choice of such a $\tilde{B}_{k+1}$ is possible. Note also that $\tilde{B}_k\sm \tilde{L}_k^{(\tilde{\beta}\tilde{r}_k)}\subset\tilde{B}_k\sm \tilde{L}_k^{(\tilde{r}'_k)}$ and the radius $\tilde{r}_{k+1}$ of $\tilde{B}_{k+1}$ satisfies
$$\tilde{r}_{k+1}=2r_{i_{k+1}}\ge2\beta r_{i_{k+1}-1}>2\beta\cdot\beta r_{i_k}/6=\tilde{\beta}\tilde{r}_k.$$
So the move of Bob's brother is legal for Game 2. Next, Alice asks her sister to choose a hyperplane neighborhood $\tilde{L}_{k+1}^{(\tilde{r}'_{k+1})}$ in $V$ according to the winning strategy for Game 2. Then Alice choose the hyperplane neighborhood $L_{i_{k+1}}^{(r'_{i_{k+1}})}$ in $W$ such that $r'_{i_{k+1}}=\beta r_{i_{k+1}}$ and $L_{i_{k+1}}=L(B_{i_{k+1}},\tilde{B}_{k+1},\tilde{L}_{k+1})$.

Let us show that the strategy constructed above guarantees a win for Alice. Since Alice's sister is playing according to the winning strategy for Game 2, we have $\bigcap_{k=0}^\infty\tilde{B}_k\subset S$. Since $B_{i_k}$ and $\tilde{B}_k$ are compatible, we have $B_{i_k}\subset P_W(\tilde{B}_k)$. It follows that
$$\bigcap_{i=0}^\infty B_i=\bigcap_{k=0}^\infty B_{i_k}\subset\bigcap_{k=0}^\infty P_W(\tilde{B}_k)=P_W\left(\bigcap_{k=0}^\infty\tilde{B}_k\right)\subset P_W(S).$$
Hence Alice wins. This proves Statement ($*$), and thus completes the proof of Lemma~\ref{L:haw-proj}. \qed

\end{document}